\documentclass[11pt]{amsart}
\usepackage{amssymb,amsmath, amsthm, amsfonts, tikz-cd}

\usepackage{graphicx}
\usepackage{listings}
\usepackage[margin=1in]{geometry}
\usepackage{lstautogobble}
\usepackage{enumerate}
\usepackage[shortlabels]{enumitem}
\usepackage{thmtools}
\usepackage{thm-restate}
\usepackage{amsthm}
\usepackage{verbatim}

\usepackage{mathtools}
\usepackage{physics}
\usepackage{color}
\usepackage{hyperref}
\usepackage[capitalise]{cleveref}
\usepackage[bottom]{footmisc}
\crefformat{equation}{(#2#1#3)}

\usepackage{float}
\restylefloat{table}

\usepackage{mathrsfs}
\setlist{  
  listparindent=\parindent,
  parsep=0pt,
}

\theoremstyle{plain}
\newtheorem{thm}{Theorem}[section]
\newtheorem{prop}[thm]{Proposition}
\newtheorem{lemma}[thm]{Lemma}
\newtheorem{cor}[thm]{Corollary}

\theoremstyle{definition}

\newtheorem{remark}[thm]{Remark}

\newtheorem{ques}{Question}[section]

\Crefname{thm}{Theorem}{Theorems}
\Crefname{prop}{Proposition}{Propositions}

\numberwithin{equation}{section} 

\DeclarePairedDelimiter{\paren}{\lparen}{\rparen}

\DeclarePairedDelimiter{\jp}{\langle}{\rangle}

\DeclareMathOperator{\supp}{supp}

\newcommand{\M}{{\mathcal{M}}}

\newcommand{\p}{{\partial}}

\renewcommand{\d}{\delta}

\newcommand{\R}{{\mathbb{R}}}
\newcommand{\C}{{\mathbb{C}}}
\newcommand{\N}{{\mathbb{N}}}

\newcommand{\T}{{\mathbb{T}}}
\newcommand{\g}{{\mathsf{g}}}
\newcommand{\G}{{\mathsf{G}}}

\newcommand{\Sc}{{\mathcal{S}}}

\renewcommand{\M}{{\mathbb{M}}}
\newcommand{\I}{\mathbb{I}}

\newcommand{\ga}{\gamma}
\newcommand{\X}{\mathsf{X}}

\newcommand{\tl}{\tilde}

\newcommand{\D}{\Delta}

\newcommand{\ph}{\phantom{=}}
\newcommand{\nn}{\nonumber}

\newcommand{\ol}{\overline}

\newcommand{\ep}{\epsilon}
\newcommand{\vep}{\varepsilon}
\newcommand{\al}{\alpha}
\newcommand{\be}{\beta}
\newcommand{\ka}{\kappa}
\newcommand{\la}{\lambda}

\newcommand{\Tc}{\mathcal{T}}

\newcommand{\indic}{\mathbf{1}}

\newcommand{\F}{{\mathcal{F}}}
\newcommand{\Dm}{|\nabla|}

\renewcommand{\P}{\mathcal{P}}

\newcommand{\om}{\omega}

\let\div\relax
\DeclareMathOperator{\div}{\mathrm{div}}

\def\XXint#1#2#3{{\setbox0=\hbox{$#1{#2#3}{\int}$ }
\vcenter{\hbox{$#2#3$ }}\kern-.6\wd0}}

\let\oldtocsection=\tocsection
 
\let\oldtocsubsection=\tocsubsection
 
\let\oldtocsubsubsection=\tocsubsubsection
 
\renewcommand{\tocsection}[2]{\hspace{0em}\oldtocsection{#1}{#2}}
\renewcommand{\tocsubsection}[2]{\hspace{1em}\oldtocsubsection{#1}{#2}}
\renewcommand{\tocsubsubsection}[2]{\hspace{2em}\oldtocsubsubsection{#1}{#2}}

\title[Global flows with random diffusion]{Global solutions of aggregation equations and other flows with random diffusion}

\author[M. Rosenzweig]{Matthew Rosenzweig}
\email{mrosenzw@mit.edu}
\author[G. Staffilani]{Gigliola Staffilani}
\email{gigliola@math.mit.edu}
\thanks{M.R. and G.S. are supported by the Simons Foundation through the Simons Collaboration on Wave Turbulence and by NSF grant DMS-2052651.}

\begin{document}
\begin{abstract}
Aggregation equations, such as the parabolic-elliptic Patlak-Keller-Segel model, are known to have an optimal threshold for global existence vs. finite-time blow-up. In particular, if the diffusion is absent, then all smooth solutions with finite second moment can exist only locally in time. Nevertheless, one can ask whether global existence can be restored by adding a suitable noise to the equation, so that the dynamics are now stochastic. Inspired by the work of Buckmaster et al. \cite{BNSW2020} showing that, with high probability, the inviscid SQG equation with random diffusion has global classical solutions, we investigate whether suitable random diffusion can restore global existence for a large class of active scalar equations in arbitrary dimension with possibly singular velocity fields. This class includes Hamiltonian flows, such as the SQG equation and its generalizations, and gradient flows, such as those arising in aggregation models. For this class, we show global existence of solutions in Gevrey-type Fourier-Lebesgue spaces with quantifiable high probability.
\end{abstract}

\maketitle

\section{Introduction}\label{sec:intro}
\subsection{Motivation}\label{ssec:intromot}
To motivate the problem addressed in this article, let us consider the two-dimensional aggregation-diffusion equation
\begin{equation}\label{eq:KS}
\begin{cases}
\p_t\theta = \div(\theta\nabla\g\ast\theta) + \nu\D\theta \\
\theta|_{t=0} = \theta^0
\end{cases}
\qquad (t,x) \in \R_+\times\R^2.
\end{equation}
Here, $\g(x) = \frac{1}{2\pi}\ln|x|$ is the Newtonian potential on $\R^2$ and $\nu\geq 0$ is the diffusion strength. It is natural to assume that $\theta^0\geq 0$ and consider solutions $\theta\geq 0$, as $\theta$ is supposed to represent a density. If $\nu >0$, then equation \eqref{eq:KS} is known as the parabolic-elliptic Patlak-Keller-Segel (PKS) equation, which is a model for the aggregation of cells by chemotaxis \cite{Patlak1953, KS1970, Nanjundiah1973}. If $\nu=0$, then the equation is the gradient flow of the Newtonian energy with respect to the 2-Wasserstein metric. The equation has been studied as a model for the evolution of vortex densities in superconductors \cite{E1994, CRS1996} and as a model for adhesion dynamics \cite{NPS2001, Poupaud2002}.  

It is a straightforward calculation that any smooth solution to \eqref{eq:KS} conserves mass, so we can unambiguously write $M=\int_{\R^2}\theta(x)dx = \int_{\R^2}\theta^0(x)dx$. Suppose that $\theta$ is a solution to \eqref{eq:KS} with finite second moment $\int_{\R^2}|x|^2\theta^t(x)dx$. Evidently this quantity is strictly positive if $\theta^t$ is not identically zero. Using integration by parts, one computes
\begin{equation}
\frac{d}{dt}\int_{\R^2}|x|^2\theta^t(x)dx = -\frac{M^2}{2\pi} + 4\nu M = \frac{M}{2\pi}\paren*{8\pi\nu-M}.
\end{equation}
Thus, if $M>8\pi\nu$, then the second moment is strictly decreasing at a linear rate. Since the second moment is nonnegative, this implies that the maximal time of existence for $\theta$ is finite. In particular, we see that if $\nu=0$, so there is no diffusion, then any nonzero, sufficiently localized classical solution to \eqref{eq:KS} must have finite lifespan \cite{JL1992}. In fact, for initial datum in $L^1$, one has a unique, global mild solution to \eqref{eq:KS} if and only if $M\leq 8\pi\nu$ \cite{Wei2018}. For the asymptotic behavior of solutions, we refer to \cite{BDP2006, CD2014} ($M<8\pi\nu$), \cite{BCM2008, GM2018} ($M=8\pi\nu$), and \cite{Velazquez2002, Velazquez2004i, Velazquez2004ii} ($M>8\pi\nu$), and references therein. In the case $\nu=0$, one has a sharp bound for the time of existence for compactly supported $L^\infty$ weak solutions to \eqref{eq:KS}, which are necessarily unique, as well as exact solutions that provide an explicit example of finite-time collapse to a nontrivial measure \cite{BLL2012}.

For the deterministic dynamics of equation \eqref{eq:KS}, we see that global existence is a nonstarter for classical solutions if the diffusion is too weak relative to the size of the initial data. But in the past two decades, there has been intense research activity on understanding how adding some noise structure (in varying forms) to deterministic equations can impact the behavior of solutions. A small, non-exhaustive sample of this research is given by \cite{dBD2002, dBD2005, FGP2010, FGP2011, DT2011, Flandoli2011, GhV2014, CG2015, BFGM2019, BNSW2020, FL2021} and references therein. Concerning equations of the form \eqref{eq:KS}, Flandoli et al. \cite{FGL2021} have shown that blow-up is \emph{delayed} in a 3D version of \eqref{eq:KS} with positive $\nu$ on $\T^d$ by adding a suitable multiplicative noise of transport type. Misiats et al. \cite{MST2021} have shown that some choices of random perturbations of equation \eqref{eq:KS} with $\nu>0$ lead to global solutions for small-mass initial data, while other choices lead to finite-time blow-up with positive probability for all initial data .

To the best of our knowledge, prior works have not shown that noise prevents finite-time blow-up, in particular for the case $\nu=0$ when all smooth, sufficiently localized solutions necessarily blow up in finite time.\footnote{In the interests of completeness, we also mention that several works (e.g. \cite{KX2016, BH2017, IXZ2021}) have investigated the suppression of finite-time blow-up in the Patlak-Keller-Segel equation by \emph{deterministic} perturbations of convective type.} Recently, the second author together with Buckmaster et al. \cite{BNSW2020} showed that adding \emph{random diffusion} leads to global solutions with positive probability for the invsicid surface quasi-geostrophic (SQG) equation with Gevrey-type initial data. Unlike the equation \eqref{eq:KS}, inviscid SQG is a Hamiltonian flow, and the long-time dynamics of classical solutions is still unresolved. In light of this result, it is natural to ask if random diffusion may somehow improve the existence theory for the equation \eqref{eq:KS}, for which solutions a priori behave very differently. Thus, we pose the following question, which the present article seeks to answer.

\begin{ques}\label{ques}
Can one restore global existence of sufficiently regular solutions to \eqref{eq:KS} by adding a suitable \emph{random diffusion}?
\end{ques}

\subsection{Problem formulation}\label{ssec:intropf}
In order to investigate the regularizing effect of random diffusion and answer \cref{ques}, let us start from more general deterministic equations of the form
\begin{equation}\label{eq:pde}
\begin{cases}
\p_t\theta + \div\paren*{\theta\M\nabla\g\ast\theta} = 0 \\
\theta|_{t=0} = \theta^0
\end{cases}
\qquad (t,x) \in \R\times\R^d.
\end{equation}
One could also include a diffusion term $-\Dm^\la\theta$, for $\la>0$, in the right-hand side (see \cref{rem:diffu} below), but we will not do so here. Above, $\M$ is a $d\times d$ constant matrix. There are several meaningful choices for $\M$. For instance, if we choose $\M$ to be $-\I$, then we get gradient flows. While if we choose $\M$ to be antisymmetric,\footnote{This case is limited to dimensions $d\geq 2$, since there is no antisymmetric matrix (i.e. scalar) in dimension 1.} then we obtain conservative/Hamiltonian flows. We assume that $\g\in\Sc'(\R^d)$ is a tempered distribution, such that the Fourier transform of $\nabla\g$ is locally integrable and satisfies the bound $|\xi\hat{\g}(\xi)| \lesssim |\xi|^{1-\ga}$ for some $0<\ga <d+1$. The model case is when $\g$ is a log or Riesz potential according to the rule
\begin{equation}\label{eq:gmodel}
\pm\begin{cases}
-\log |x|, & {\ga = d} \\
|x|^{\ga-d}, & {\ga \in (0,d+1)\setminus\{d\}}.
\end{cases}
\end{equation}
The choice of sign determines whether the potential is repulsive ($+$) or attractive ($-$). When $\gamma=2$, \eqref{eq:gmodel} is a constant multiple of the Coulomb/Newtonian potential. We refer to the ranges $\ga>2$ and $\ga<2$ as sub-Coulombic and super-Coulombic, respectively.

The general equation \eqref{eq:pde} encompasses a wide class of physical models. Focusing first on the conservative case, in which $\M$ is antisymmetric, the most notable examples are in dimension 2. If $\M$ is rotation by $\frac{\pi}{2}$ and $\g(x)=-\frac{1}{2\pi}\log|x|$ is the Coulomb potential, then \eqref{eq:pde} becomes the incompressible Euler vorticity equation (for instance, see \cite[Section 1.2] {MP2012book} or \cite[Chapter 2]{MB2002}). If $\g(x)=C|x|^{-1}$, then equation \eqref{eq:pde} becomes the inviscid SQG equation, which models the motion of a rotating stratified fluid with small Rosby and Froude numbers in which potential vorticity is conserved \cite{Pedlosky2013, CMT1994, HPGS1995, Resnick1995}. More generally, choosing $\g(x) = C_\ga |x|^{\ga-2}$, for $0<\ga<2$, leads to the generalized SQG (gSQG) family of equations \cite{PHS1994spec, CCCGW2012}, for which the Euler vorticity equation is the $\ga\rightarrow 2^-$ limit. While global-well posedness is known for classical \cite{Wolibner1933, Holder1933} and weak \cite{Yudovich1963} solutions to the Euler case, the global existence of smooth solutions to the gSQG equation is a major open problem---it is only known if one adds suitably strong diffusion to \eqref{eq:pde} (e.g. see \cite{CW1999QG, KNV2007gwp, CV2010, CV2012nmp}). We refer the reader to \cite{CMT1994, Resnick1995, CF2002, Gancedo2008, BSV2019, CGI2019, BCCK2020, GP2021gsqg, HK2021} and references therein for more information on the well-posedness and long-time dynamics of the gSQG equation.

In the gradient-flow case, in which $\M=-\I$, equation \eqref{eq:pde} has been studied for several applications in addition to the aforementioned ones of adhesion dynamics, chemotaxis, and vortices in superconductors. To name a few: materials science \cite{HP2006}, cooperative control \cite{GP2003}, granular flow \cite{BCP1997, BCCP1998, Toscani2000, CMV2006}, phase segregation in lattice matter models \cite{GL1997, GL1998, GLM2000}, and swarming models \cite{MEk1999, MEkBS2003, TB2004, TBL2006}. Several works have focused on the well-posedness and long-time dynamics. We recount some of the results for the model interaction \eqref{eq:gmodel}, which is sometimes called a fractional porous medium equation. In particular, in the repulsive case $\ga=2$, global existence, uniqueness, and asymptotic behavior of nonnegative classical and $L^\infty$ weak solutions are known \cite{LZ2000, AS2008, BLL2012, SV2014}. The case $2<\ga < d+1$ is easier and follows by the same arguments \cite[Section 4]{CCH2014} (see also \cite{BLR2011} for an $L^p$ well-posedness result). For $0<\ga<2$, local well-posedness of nonnegative classical solutions is known \cite{CJ2021} and global existence, regularity, and asymptotic behavior of certain nonnegative weak solutions are known \cite{CV2011, CV2011asy, CSV2013, CV2015, BIK2015, CHSV2015, LMS2018}. To our knowledge, these weak solutions are only known to be unique if $d=1$ \cite{BKM2010}. It is also an open problem whether classical solutions are global if $0<\ga<2$. If one allows for mixed-sign solutions, then the repulsive and attractive equations are equivalent by multiplication by $-1$. \cite{BLL2012} has established well-posedness, in particular maximal time of existence, for  compactly supported classical and $L^\infty$ weak solutions in the $\ga=2$ case. In particular, nonnegative classical and $L^\infty$ weak solutions in the $\ga=2$ case are known to blow up in finite time, as remarked at the beginning of the introduction. \cite{MZ2005} has shown the existence of global renormalized solutions in the sense of DiPerna-Lions \cite{DL1989ode, DL1989b}. We also mention the works \cite{MZ2005, AMS2011, Mainini2012} for an equation arising in vortex superconductivity, which reduces to the repulsive $\ga=2$ case of equation \eqref{eq:pde} when one considers nonnegative solutions.

As a unifying perspective, the equation \eqref{eq:pde} may be viewed as an effective description of first-order mean-field dynamics of the form
\begin{equation}
\begin{cases}
\dot{x}_i^t = \displaystyle \frac{1}{N}\sum_{1\leq j\leq N : j\neq i} \M\nabla\g(x_i^t-x_j^t) \\
x_i^t|_{t=0} = x_i^0
\end{cases}
\qquad i\in\{1,\ldots,N\}
\end{equation}
in the limit as the number of particles $N\rightarrow\infty$. The mathematical validity of this description has been actively studied over the years \cite{Dobrushin1979, Sznitman1991, Hauray2009, CCH2014, Jab2014, Golse2016ln, Duerinckx2016, JW2018, BO2019, BJW2019edp, Serfaty2020, NRS2021}.

Inspired by the aforementioned work of Buckmaster et al. \cite{BNSW2020}, which in turn was inspired by earlier work of Glatt-Holtz and Vicol \cite{GhV2014}, we propose adding a random diffusion term to \eqref{eq:pde} by considering the stochastic partial differential equation
\begin{equation}\label{eq:spde}
\begin{cases}
\p_t\theta +\div(\theta\M\nabla\g\ast\theta)  = \nu(1+\Dm^s)\theta\dot{W}^t \\
\theta|_{t=0} = \theta^0
\end{cases}
\qquad (t,x)\in\R\times\R^d.
\end{equation}
Here, $s\geq 0$ (we will determine further restrictions later), $W$ is a standard real Brownian motion, and the stochastic differential should be interpreted in the It\^o sense. The $\dot{}$ superscript formally denotes differentiation with respect to time. We note that our choice of random diffusion differs from that of \cite{BNSW2020}, which used the fractional Laplacian $\Dm^s = (-\D)^{\frac{s}{2}}$. In that article, the authors work in the periodic setting of $\T^2$, and after modding out by the mass of the solution, which is conserved, homogeneous and inhomogeneous Sobolev spaces are equivalent. This equivalence fails on $\R^d$, and therefore the fractional Laplacian creates problems at low frequency, as will become clear to the reader in \Cref{sec:lwp,sec:glob} (see \cref{rem:per} for further comments). Accordingly, we opt to add an inhomogeneity to rectify this issue. We emphasize that our choice of random perturbation differs from the aforementioned prior works \cite{FGL2021, MST2021} on stochastic PKS equations, which did \emph{not} consider random diffusion as in \eqref{eq:spde}.

A priori, it is not clear how to interpret the SPDE \eqref{eq:spde}. Moreover, it is not clear that the stochastic term in the right-hand side is regularizing since $W^t$ does not have definite sign. Formally, suppose that we have a solution $\theta$ to \eqref{eq:spde}, and let us set $\mu^t\coloneqq e^{-\nu W^t(1+\Dm^s)}\theta^t$, where for each realization of $W$, $\Gamma^t\coloneqq e^{-\nu W^t(1+\Dm^s)}$ is the Fourier multiplier with symbol $e^{-\nu W^t(1+|\xi|^s)}$. As in \cite{BNSW2020}, to compute the equation satisfied by $\mu$, we formally use the Fourier transform together with It\^o's lemma to obtain
\begin{align}
\p_t\mu &= \Gamma\p_t\theta + \p_t\Gamma\theta + \p_t\comm{\Gamma}{\theta} \nn\\
&= \Gamma\paren*{-\div(\theta\M\nabla\g\ast\theta) + \nu(1+\Dm^s)\theta\dot{W}} + \paren*{-\nu(1+\Dm^s)\Gamma\dot{W} + \frac{\nu^2}{2}(1+\Dm^{s})^2\Gamma}\theta \nn\\
&\quad -\nu^2(1+\Dm^{s})^2\Gamma\theta \nn\\
&= - \div\Gamma\paren*{\Gamma^{-1}\mu\M\nabla\g\ast\Gamma^{-1}\mu}- \frac{\nu^2}{2}(1+\Dm^{s})^2\mu. \label{eq:rpde}
\end{align}
Above, $[\Gamma,\theta]$ denotes the quadratic covariation of the processes $\Gamma$ and $\theta$. Also, we have implicitly used that $\Gamma^t$ and $\Dm^s$ commute, both being Fourier multipliers. Observe that \eqref{eq:rpde} is a random PDE which may be interpreted pathwise (i.e. for fixed realization of $W$, which almost surely is a locally continuous path on $[0,\infty)$). Additionally, thanks to the nontrivial quadratic variation of Brownian motion, we have gained a diffusion term in this equation. Rather than deal with the original equation \eqref{eq:spde}, we shall base our mathematical interpretation on \eqref{eq:rpde}.

\begin{remark}\label{rem:Strat}
One may wonder why we choose the It\^o formulation in \eqref{eq:spde} as opposed to the Stratonovich formulation
\begin{equation}
\p_t\theta +\div(\theta\M\nabla\g\ast\theta)  = \nu(1+\Dm^s)\theta\circ\dot{W}^t
\end{equation}
which is formally equivalent to the It\^o equation
\begin{equation}\label{eq:strat}
\p_t\theta + \div(\theta\M\nabla\g\ast\theta) = \nu(1+\Dm^s)\theta\dot{W}^t + \frac{\nu^2}{2}(1+\Dm^{s})^2\theta.
\end{equation}
Suppose we define $\mu^t \coloneqq e^{-\nu W^t(1+\Dm^s)}\theta^t$ as before. Then again using It\^o's lemma, we find
\begin{align}
\p_t\mu &= \Gamma\p_t\theta + \p_t\Gamma\theta + \p_t\comm{\Gamma}{\theta} \nn\\
&=\Gamma\paren*{-\div(\theta\M\nabla\g\ast\theta) + \nu(1+\Dm^s)\theta\dot{W}^t + \frac{\nu^2}{2}(1+\Dm^{s})^2\theta} \nn \\
&\ph + \paren*{-\nu(1+\Dm^s)\Gamma\dot{W} + \frac{\nu^2}{2}(1+\Dm^{s})^2\Gamma}\theta  -\nu^2(1+\Dm^{s})^2\Gamma\theta \nn\\
&= - \Gamma\div(\Gamma^{-1}\mu\M\nabla\g\ast\Gamma^{-1}\mu).
\end{align}
No longer do we gain a fractional diffusion term, which, as we shall see below, is fatal to our arguments. The preceding conclusion is to be expected. Indeed, if $W$ is a $C^1$ path, then by ordinary calculus, $\p_t\mu = -\nu(1+\Dm^s)\Gamma\theta\dot{W} + \Gamma\p_t\theta$; and the Stratonovich formulation is precisely chosen to preserve the ordinary rules of calculus.
\end{remark}

\subsection{Statement of main results}\label{ssec:intromr}
We now present our main theorem. We assume that we have a standard real Brownian motion $\{W^t\}_{t\geq 0}$ defined on a filtered probability space $(\Omega, \F, \{\F^t\}_{t\geq 0},\P)$ satisfying all the usual assumptions. Given $\al,\be,\nu>0$, consider the event
\begin{equation}\label{eq:Omsetdef}
\Omega_{\al,\be,\nu} \coloneqq \{\omega\in\Omega : \al+\beta t - \nu W^t(\omega) \geq 0 \quad \forall t\in [0,\infty)\} \subset\Omega.
\end{equation}
It is well-known (see \cite[Proposition 6.8.1]{Resnick1992}) that $\P(\Omega_{\al,\be,\nu}) = 1 - e^{-\frac{2\al\be}{\nu^2}}$. We recall the Fourier-Lebesgue norms $\|\cdot\|_{\hat{W}^{\ka,r}}$ (see \cref{ssec:preSob}).

\begin{thm}\label{thm:main}
Let $d\geq 1$, $0<\ga<d+1$, $\frac{1}{2}<s\leq 1$. If $\ga>1$, also suppose that we are given $1<q<\frac{d}{\ga-1}$. Given $\al,\be,\nu>0$, suppose that $\beta<\frac{\nu^2}{2}$.

Suppose first that $\ga\leq 1$. If $s$ is sufficiently large depending on $\ga$, then there exists an $r_0\geq 1$ depending on $d,\ga,s$, such that the following holds. For any $1\leq r \leq r_0$ and any $\sigma>0$ sufficiently large depending on $d,\ga,r,s$, there is a constant $C>0$ depending only on $d,\ga,r,s,\sigma$, such that for initial datum $\mu^0$ satisfying
\begin{equation}\label{eq:idcon}
\|e^{(\al+\ep)(1+\Dm^{s})}\mu^0\|_{\hat{W}^{\sigma s,r}} < \frac{\nu^2-2\be}{C|\M|}
\end{equation}
and any path in $\Omega_{\al,\be,\nu}$, there exists a unique global solution $\mu\in C([0,\infty); \hat{W}^{\sigma s,r})$ to equation \eqref{eq:rpde} with initial datum $\mu^0$. Moreover, for $\phi^t\coloneqq \al+\be t$, the function
\begin{equation}\label{eq:Gevfunc}
t\mapsto \|e^{(\phi^t+\ep)(1+\Dm^s)}\mu^t\|_{\hat{W}^{\sigma s,r}}
\end{equation}
is strictly decreasing on $[0,\infty)$.

Now suppose that $\ga>1$. For $\sigma_r,\sigma_q>0$ sufficiently large depending on $d,\ga,s,r,q$, there is a constant $C>0$ depending only on $d,\ga,r,q,s,\sigma_r,\sigma_q$, such that for initial datum $\mu^0$ satisfying
\begin{equation}\label{eq:qidcon}
\|e^{(\al+\ep)(1+\Dm^{s})}\mu^0\|_{\hat{W}^{\sigma_r s,r}} + \|e^{(\al+\ep)(1+\Dm^{s})}\mu^0\|_{\hat{W}^{\sigma_q s,\frac{2q}{q-1}}} < \frac{\nu^2-2\be}{C|\M|}
\end{equation}
and any path in $\Omega_{\al,\be,\nu}$, there exists a unique global solution $\mu\in C([0,\infty); \hat{W}^{\sigma_r s,r} \cap \hat{W}^{\sigma_q s, q})$ to equation \eqref{eq:rpde} with initial datum $\mu^0$. Moreover, the function
\begin{equation}\label{eq:qGevfunc}
t\mapsto \|e^{(\phi^t+\ep)(1+\Dm^s)}\mu^t\|_{\hat{W}^{\sigma_r s,r}} + \|e^{(\phi^t+\ep)(1+\Dm^s)}\mu^t\|_{\hat{W}^{\sigma_q s,\frac{2q}{q-1}}}
\end{equation}
is strictly decreasing on $[0,\infty)$.
\end{thm}

To the best our knowledge, our result is the first demonstration that a random diffusion term can lead to global solutions for equations which, without any diffusion, necessarily blow up in finite time. This provides an affirmative answer to \cref{ques}. Furthermore, \cref{thm:main} substantially generalizes the prior work of Buckmaster et al. \cite[Theorem 1.1]{BNSW2020}, which was limited to the SQG case $\M$ equals rotation by $\frac{\pi}{2}$ and $\hat{\g}(\xi) = |\xi|^{-1}$, corresponding to a conservative/Hamiltonian flow. In particular, our result covers the full range of interactions in the model case \eqref{eq:gmodel} and also allows for interactions (e.g. $d<\ga<d+1$) which may not be singular in physical space near the origin but have very slow decay or even growth at $\infty$.

We do not say anything here about the asymptotic behavior of the solutions we construct, only that they are global. It would be interesting to give an asymptotic description of the solution as $t\rightarrow\infty$, valid at least with positive probability. Indeed, the reader will recall from the beginning of the introduction that such a description is known for the deterministic PKS equation. We hope to address this question in future work.

Before transitioning to discuss the proof of \cref{thm:main}, let us record a few remarks on the statement of and assumptions behind the theorem.

\begin{remark}\label{rem:path}
The solutions in \cref{thm:main} are \emph{pathwise}. More precisely, there is a good set $\Omega_{\al,\be,\nu}$, defined above, of realizations of the Brownian motion, such that for any $\omega\in\Omega_{\al,\be,\nu}$ and with $W(\omega): [0,\infty)\rightarrow \R$, we have a unique global solution $\mu(\omega): [0,\infty) \rightarrow \R$ to equation \eqref{eq:rpde}. We can then define a notion of solution to the original equation \eqref{eq:spde} by setting $\theta^t(\omega) \coloneqq e^{\nu W^t(\omega)(1+\Dm^s)}\mu^t(\omega)$. Since for any $\om\in\Omega_{\al,\be,\nu}$, we have $\phi^t - \nu W^t\geq 0$, it follows from the definition of the Fourier-Lebesgue norm that
\begin{equation}
\|\theta^t(\omega)\|_{\hat{W}^{\sigma s,p}} \leq \|e^{\phi^t(1+\Dm^s)}\mu^t(\omega)\|_{\hat{W}^{\sigma s,p}}.
\end{equation}
for any $1\leq p\leq \infty$.

One may interpret \cref{thm:main} as follows (cf. \cite[Remark 5.3]{BNSW2020}). Fixing $\al,\ep,\nu$ and given an initial condition $\mu^0$ such that
\begin{equation}
E\coloneqq \|e^{(\al+\ep)(1+\Dm^s)}\mu^0\|_{\hat{W}^{\sigma_r s,r}} + \|e^{(\al+\ep)(1+\Dm^s)}\mu^0\|_{\hat{W}^{\sigma_q s,\frac{2q}{q-1}}}\indic_{\ga>1} < \infty,
\end{equation}
we can choose $\beta< \frac{\nu^2-CE|\M|}{2}$, where $C$ is the constant from condition \eqref{eq:idcon} or \eqref{eq:qidcon}. Then with probability at least
\begin{equation}
P = 1- \exp(-\frac{\al(\nu^2-CE|\M|)}{\nu^2}),
\end{equation}
there is a pathwise unique global solution $\mu \in C([0,\infty); \hat{W}^{\sigma_r s, r} \cap \hat{W}^{\sigma_q s, \frac{2q}{q-1}})$ to equation \eqref{eq:rpde} with initial datum $\mu^0$, such that the function \eqref{eq:qGevfunc} is strictly decreasing on $[0,\infty)$. 
\end{remark}

\begin{remark}\label{rem:size}
So as to make the result as accessible as possible, we have opted not to include in the statement of \cref{thm:main} the explicit relations the parameters, such as $d,\ga,s,r_0,\sigma$, have to satisfy in order for the theorem to apply. These relations are explicitly worked out in \Cref{sec:lwp,sec:glob} during the proofs of \cref{prop:lwp,prop:mon}. Here and throughout this article, the reader should keep in mind that the most favorable choices for $s,r$ are $s=1$ and $r=1$.
\end{remark}

\begin{remark}\label{rem:IDcon}
The conditions \eqref{eq:idcon}, \eqref{eq:qidcon} allows for initial data of arbitrarily large mass. Indeed, focusing on the $\ga\leq 1$ case, suppose that $\hat{\mu}^0\in C_c^\infty$ and $\hat{\mu}^0(0) = M$, for given $M$. Let $L= \sup_{\xi\in \supp(\hat{\mu}^0)} |\xi|$. Then
\begin{equation}
\|e^{(\al+\ep)(1+\Dm^s)}\mu^0\|_{\hat{W}^{\sigma s,r }} \leq C_d\paren*{1+L}^{\sigma s} L^{\frac{d}{r}}e^{(\al+\ep)(1+L^s)}\|\hat{\mu}^0\|_{L^\infty}.
\end{equation}
Taking $\beta = \vep\frac{\nu^2}{2}$, for given $\vep\in (0,1)$, and requiring
\begin{equation}\label{eq:nucon}
\frac{C_dC|\M|\paren*{1+L}^{\sigma s} L^{\frac{d}{r}}e^{(\al+\ep)(1+L^s)}\|\hat{\mu}^0\|_{L^\infty}}{(1-\vep)} < \nu^2,
\end{equation}
we see that \eqref{eq:idcon} holds. For fixed $\nu$, we can make the left-hand side of the preceding inequality arbitrarily small by letting $L\rightarrow 0^+$. While for given $L$, we can take $\nu$ arbitrarily large so that \eqref{eq:nucon} holds. The latter case is reminiscent of the mass-diffusion threshold we saw earlier for the PKS equation. 
\end{remark}

\begin{remark}\label{rem:nu}
Although only the ratio $\frac{2\beta}{\nu^2}$ appears in the value of $\P(\Omega_{\al,\be,\nu})$, which might suggest to the reader that one can send $\beta,\nu\rightarrow 0^+$ while fixing $\frac{2\beta}{\nu^2}$, we emphasize that the assumption $0<\beta<\frac{\nu^2}{2}$ is crucial. Indeed, our argument for showing local well-posedness fails if $\beta\geq \frac{\nu^2}{2}$, and the requirement $\beta<\frac{\nu^2}{2}$ appears when showing the functions \eqref{eq:Gevfunc}, \eqref{eq:qGevfunc} are strictly decreasing. Furthermore, if $\nu,\beta\rightarrow 0^+$, then the right-hand sides of the initial datum conditions \eqref{eq:idcon}, \eqref{eq:qidcon} are tending to zero, which implies that only $\mu^0\equiv 0$ would satisfy these conditions in the limit.

Additionally, one might think that by increasing $\nu$, the diffusion becomes stronger and therefore one should get a ``better'' result. But $\P(\Omega_{\al,\be,\nu})$ evidently decreases to zero as $\nu\rightarrow\infty$, assuming $\beta$ is fixed. The reason has to deal with the resulting growing variance of $\nu W^t$ appearing in the definition of $\Gamma^t$, which requires a large value of $\beta$ to be absorbed by the exponential weight in our function spaces.
\end{remark}

\begin{remark}\label{rem:per}
\cref{thm:main} is also valid if $\R^d$ is replaced by $\T^d$. In fact, since Fourier space is discrete on the torus, we do not have the same issues at low frequency as in the setting of $\R^d$, and therefore one can replace equation \eqref{eq:spde} with
\begin{equation}\label{eq:homspde}
\p_t\theta + \div(\theta\M\nabla\g\ast\theta) = \nu\Dm^s\theta \dot{W}^t.
\end{equation}
An elementary computation reveals that solutions conserve mass and therefore one may quotient out the mass by assuming it is zero. As a result, the zero Fourier mode vanishes and one has an equivalence of homogeneous and inhomogeneous Sobolev norms. Working with \eqref{eq:homspde} simplifies the proof greatly, as the two-tiered norm for $\ga>1$ becomes unnecessary.
\end{remark}

\begin{remark}\label{rem:diffu}
\cref{thm:main} is still valid if one adds a deterministic diffusion term $-\chi\Dm^{\la}\theta$ to the right-hand side of \eqref{eq:spde}, for $\chi,\la>0$, which leads to \eqref{eq:rpde} being replaced by
\begin{equation}
\p_t\mu + \div\Gamma\paren*{\Gamma^{-1}\mu\M\nabla\g\ast\Gamma^{-1}\mu} = -\paren*{\frac{\nu^2}{2}(1+\Dm^{s})^2 + \chi\Dm^{\la}}\mu.
\end{equation}
Since a deterministic diffusion term only makes the circumstances for global existence more favorable, we have opted not to include this term.
\end{remark}

\subsection{Comments on proof}\label{ssec:introprf}
We briefly comment on the proof of \cref{thm:main}. In light of the success of \cite{BNSW2020} in showing that adding random diffusion to the inviscid SQG equation leads, with high probability, to global solutions, and that the SQG equation is a special case of \eqref{eq:pde}, we are guided by the approach of the cited work. There are two main steps:
\begin{enumerate}[(1)]
\item\label{item:StepLWP}
Local well-posedness via contraction mapping argument,
\item\label{item:StepMon}
Monotonicity of the Gevrey norm via energy estimate.
\end{enumerate}
As discussed below, repeating the proof of \cite{BNSW2020} in our more general context would fail due to issues at low frequency related to working on $\R^d$, as opposed to $\T^d$, and issues at high frequency stemming from the singularity of our interactions. Several new ideas are consequently needed.

\medskip
Step \ref{item:StepLWP}, carried out in \cref{sec:lwp}, proceeds by rewriting the equation \eqref{eq:rpde} in mild form (see \eqref{eq:mild}) which is amenable to a contraction mapping argument for short times. The main difficulty is estimating the nonlinear term in the scale of Gevrey-type spaces defined in \cref{ssec:lwpGev}---the exponential weights of which are used to absorb the $\Gamma$ operators---in which we want to construct solutions. In particular, the velocity field $\M\nabla\g\ast\mu$ can be singular compared to the regularity of the scalar $\mu$, as opposed to of the same order in the SQG case of \cite{BNSW2020}, which requires carefully balancing the derivatives in the nonlinearity against the diffusion.

It turns out that using $L^2$-based function spaces, as in \cite{BNSW2020}, leads to a restriction on $\ga$ that scales linearly in the dimension $d$, which would then limit us to strictly sub-Coulombic interactions $\g$ in dimensions $d\geq 4$. One of our new insights is to instead consider Gevrey-Fourier-Lebesgue hybrid spaces (see \eqref{eq:Gdef}), which of course include the function spaces of \cite{BNSW2020} as a special case. In particular, our new spaces behave well with respect to Sobolev embedding when the integrability exponent $r\rightarrow 1^+$, becoming an algebra at $r=1$.

Another challenge in the local well-posedness step is the singularity near the origin of the Fourier transform $\hat{\g}$ when $\ga$ is large. In particular, for $\ga>1$, $|\xi\hat{\g}(\xi)|$ may blow up as $|\xi|\rightarrow 0$. Dealing with this issue requires using a two-tiered function space, compared to the case $\ga\leq 1$. More precisely, at high frequency, we need our functions to be in an exponentially-weighted Fourier-Lebesgue space with high regularity index and low integrability exponent; while at low frequency, we need our functions to be in a similarly weighted space with low regularity index and high integrability exponent. This leads us to the multi-parameter scale of spaces $\X_{\phi,\ga}^{\sigma,\tl{\sigma},r,\tl{r},}$ introduced in \eqref{eq:Xdef} (more generally, see \cref{sec:lwp}).

After some paraproduct analysis and a fair amount of algebra to determine what conditions all the various parameters have to satisfy, we prove \cref{prop:lwp}, which asserts local well-posedness in the class of solutions satisfying
\begin{equation}\label{eq:introsup}
\sup_{0\leq t\leq T} \paren*{\|e^{(\phi^t+\ep)(1+\Dm^s)}\mu^t\|_{\hat{W}^{\sigma_{lwp},r}}  + \|e^{(\phi^t+\ep)(1+\Dm^s)}\mu^t\|_{\hat{W}^{0,\frac{2q}{q-1}}}\indic_{\ga>1}} < \infty,
\end{equation}
for some $\sigma_{lwp} < \sigma$. Here, $\indic_{(\cdot)}$ denotes the indicator function for the condition $(\cdot)$. Although the Sobolev index $\sigma_{lwp}$ is strictly less than that of the initial datum, we will later improve it to $\sigma$ through a bootstrap argument.

\medskip
Step \ref{item:StepMon}, carried out in \cref{sec:glob}, consists of upgrading the local solution from step \ref{item:StepLWP} to a global solution and also upgrading the Sobolev index from $\sigma_{lwp}$ to $\sigma$. The original idea of \cite[Proposition 4.1]{BNSW2020}, modified to our setting and presented for the $r=2$ case, is to prove an inequality for the time derivative of the ``energy'' $\|e^{\phi^t(1+\Dm^s)}\mu^t\|_{H^{\sigma s}}^2$, which shows that this quantity is strictly decreasing on an interval $[0,T]$, provided it is not too large at initial time and that $\|e^{\phi^t(1+\Dm^s)}\mu^t\|_{H^{(\sigma+1)s}}$ remains finite on the same interval. With this type of conditional monotonicity result, the authors of that work could exploit the fact that the initial datum belongs to a space with higher Gevrey index $\al+\ep$ in order to iteratively extend the lifespan of the solution, losing a decreasing fraction of $\ep$ along each step of the iteration. Note that in their work the Sobolev index from the local well-posedness does not change.

Since we deal with $\ga$ that are more singular than in \cite{BNSW2020} and step \ref{item:StepLWP} only gives local solutions in a rougher space than that claimed in the statement of \cref{thm:main}, we need a more sophisticated argument. Moreover, we need to work in our scale of Fourier-Lebesgue spaces, with the auxiliary space if $\ga>1$. We prove a similar conditional monotonicity result for the energy
\begin{equation}
\|e^{(\phi^t+\ep')(1+\Dm^s)}\mu^t\|_{\hat{W}^{\sigma_{r} s,r}}^r + \|e^{(\phi^t+\ep')(1+\Dm^s)}\mu^t\|_{\hat{W}^{\sigma_{q} s,\frac{2q}{q-1}}}^{\frac{2q}{q-1}}\indic_{\ga>1},
\end{equation}
for any $0\leq \ep'\leq \ep$, assuming $\sigma_{r},\sigma_{q}$ are sufficiently large depending on $d,\ga,s,r,q$. Similar to step \ref{item:StepLWP}, the bulk of the labor consists of paraproduct analysis for the nonlinearity and determining the set of conditions that the parameters $d,\ga,s,r,q,\sigma_{r},\sigma_{q}$ have to satisfy in order for the paraproduct analysis to be valid. In order to access the monotonicity result, since $\sigma_r > \sigma_{lwp}$ and $\sigma_q>0$, we exploit the higher Gevrey index of the initial datum together with an embedding lemma (see \cref{lem:Gemb}) to conclude that if \eqref{eq:introsup} holds, then for any $0\leq \ep'<\ep$ and $\sigma_r,\sigma_q\in\R$,
\begin{equation}
\sup_{0\leq t\leq T}\paren*{\|e^{(\phi^t + \ep')(1+\Dm^s)}\mu^t\|_{\hat{W}^{\sigma_r s,r}} + \|e^{(\phi^t+\ep')(1+\Dm^s)}\mu^t\|_{\hat{W}^{\sigma_q s,\frac{2q}{q-1}}}} < \infty
\end{equation}
also holds. We then obtain global existence by a lemma (see \cref{lem:lspan}) which quantifies the improvement in the lifespan of the solution as we decrease $\ep'$. Finally, we conclude global existence and monotonicity also hold with $\ep'=\ep$ by essentially Fatou's lemma.

\subsection{Organization of article}\label{ssec:introorg}
We close the introduction by outlining the remaining body of the article. In \cref{sec:pre}, we introduce the basic notation of the article and review some frequently used facts from Fourier analysis. In \cref{sec:lwp}, we begin (\cref{ssec:lwpGev}) with our class of Gevrey-Fourier-Lebesgue spaces and their properties and then (\cref{ssec:lwpcm}) show the local well-posedness of the Cauchy problem for equation \eqref{eq:rpde}.  In \cref{sec:glob}, we first (\cref{ssec:globmon}) show the monotonicity property of the Gevrey norm. We then (\cref{ssec:globmr}) use this property together with the local theory from \cref{sec:lwp} in order to prove our main result, \cref{thm:main}.

\subsection{Acknowledgments}
The first author thanks Sylvia Serfaty for helpful comments on the relevance of equation \eqref{eq:pde} for vortices in superconductors. Both authors gratefully acknowledge the hospitality of the Institute for Computational and Experimental Research in Mathematics (ICERM) where the manuscript for this project was completed during the ``Hamiltonian Methods in Dispersive and Wave Evolution Equations'' semester program.

\section{Preliminaries}\label{sec:pre}
\subsection{Notation}\label{ssec:prenot}
Given nonnegative quantities $A$ and $B$, we write $A\lesssim B$ if there exists a constant $C>0$, independent of $A$ and $B$, such that $A\leq CB$. If $A \lesssim B$ and $B\lesssim A$, we write $A\sim B$. To emphasize the dependence of the constant $C$ on some parameter $p$, we sometimes write $A\lesssim_p B$ or $A\sim_p B$. We denote the natural numbers excluding zero by $\N$ and including zero by $\N_0$. Similarly, we denote the positive real numbers by $\R_+$.

The Fourier and inverse transform of a function $f:\R^d\rightarrow\C^m$ are defined according to the convention
\begin{equation}
\begin{split}
\hat{f}(\xi) = \F(f)(\xi) &\coloneqq \int_{\R^d}f(x)e^{-ix\cdot\xi}dx,\\
\check{f}(x) = \F^{-1}(f)(x) &\coloneqq (2\pi)^{-d}\int_{\R^d}f(\xi)e^{i\xi\cdot x}d\xi.
\end{split}
\end{equation}
In the case $m>1$, the notation should be understood component-wise. Given a function $m:\R^d\rightarrow\C^m$, we use the notation $m(\nabla)$ to denote the $\C^m$-valued Fourier multiplier with symbol $m(\xi)$. In the particular, the notation $\Dm = (-\D)^{\frac{1}{2}}$ denotes the Fourier multiplier with symbol $|\xi|$ and $\jp{\nabla} \coloneqq (1+|\nabla|^2)^{1/2}$ denotes the multiplier with Japanese bracket symbol $(1+|\xi|^2)^{1/2}$.

\subsection{Sobolev embedding}\label{ssec:preSob}
For the reader's convenience, we state and prove an elementary Sobolev embedding tailored to the Fourier analysis of \Cref{sec:lwp,sec:glob}. To the state the lemma, we recall that the Bessel potential space $W^{s,p}$ is defined by
\begin{equation}
\|f\|_{W^{s,p}} \coloneqq \|\jp{\nabla}^{s}f\|_{L^p}, \qquad s\in\R, \ p\in (1,\infty)
\end{equation}
and the Fourier-Lebesgue space $\hat{W}^{s,p}$ is defined by
\begin{equation}
\|f\|_{\hat{W}^{s,p}} \coloneqq \|\jp{\cdot}^s\hat{f}\|_{L^p}, \qquad s\in\R, \ p\in [1,\infty].
\end{equation}
For $p=2$, these two spaces coincide by Plancherel's theorem, and, following standard notation, we shall write $H^s$. When $s=0$, we shall also adopt the notation $\hat{W}^{0,r} = \hat{L}^r$.

\begin{lemma}\label{lem:Sob}
If $1\leq p <r\leq\infty$, then
\begin{equation}
\|f\|_{\hat{W}^{s,p}} \lesssim_{d,p,r} \|f\|_{\hat{W}^{(s+\frac{d(r-p)}{rp})+,r}},
\end{equation}
where the notation $(\cdot)+$ means $(\cdot)+\vep$, for any $\vep>0$, with the implicit constant then depending on $\vep$ and possibly blowing up as $\vep\rightarrow 0^+$. If $2\leq p\leq \infty$, then
\begin{equation}
\|f\|_{\hat{W}^{s,p}} \lesssim_{d,p} \|f\|_{W^{s,\frac{p}{p-1}}}.
\end{equation}
\end{lemma}
\begin{proof}
Fix $r> 1$ and let $1\leq p<r$. By H\"older's inequality,
\begin{align}
\|f\|_{\hat{W}^{s,p}} = \|\jp{\cdot}^{s}\hat{f}\|_{L^p} = \|\jp{\cdot}^{-\d}\jp{\cdot}^{s+\d}\hat{f}\|_{L^{p}} \lesssim_d \|\jp{\cdot}^{-\d}\|_{L^{\frac{rp}{r-p}}} \|f\|_{\hat{W}^{r,s+\d}}.
\end{align}
The first factor is finite provided that $\frac{r p\d}{r-p} > d \Longleftrightarrow \d> \frac{d(r-p)}{rp}$.

Now suppose $2 \leq p\leq\infty$. If $\jp{\nabla}^s f\in L^{\frac{p}{p-1}}$, then by the Hausdorff-Young inequality $\mathcal{F}(\jp{\nabla}^s f)(\xi) = \jp{\xi}^s \hat{f}(\xi)$ belongs to $L^p$ and
\begin{equation}
\|f\|_{\hat{W}^{s,p}} = \|\jp{\cdot}^s \hat{f}\|_{L^{p}} \lesssim_{d,p}  \|\jp{\nabla}^s f\|_{L^{\frac{p}{p-1}}} = \|f\|_{W^{s,\frac{p}{p-1}}}.
\end{equation}
\end{proof}

\section{Local well-posedness}\label{sec:lwp}
We investigate the local well-posedness of the Cauchy problem
\begin{equation}\label{eq:cp}
\begin{cases}
\p_t\mu + \div\Gamma\paren*{\Gamma^{-1}\mu(\M\nabla\g\ast\Gamma^{-1}\mu)} + \frac{\nu^2}{2}(1+\Dm^{s})^2\mu =0\\
\mu|_{t=0} = \mu^0.
\end{cases}
\end{equation}
Set $A\coloneqq (1+\Dm^{s})^2$. It will be convenient to introduce the bilinear operator
\begin{equation}\label{eq:Bdef}
B(f,g) \coloneqq \div\Gamma\paren*{\Gamma^{-1}f(\M\nabla\g\ast\Gamma^{-1}g)}.
\end{equation}
The reader should note that $B$ itself depends on time through $\Gamma$, and, when necessary, we shall make explicit this time dependence by writing $B^t(f,g)$. We rewrite \eqref{eq:cp} in mild form
\begin{equation}\label{eq:mild}
\mu^t = e^{-\frac{t\nu^2}{2}A}\mu^0 - \int_0^t e^{-\frac{(t-\tau)\nu^2}{2}A} B^\tau(\mu^\tau,\mu^\tau)d\tau.
\end{equation}

In order to perform a contraction mapping argument based on the mild formulation \eqref{eq:mild}, we use a generalization of the scale of Gevrey function spaces from \cite{BNSW2020} (see also \cite{FT1989}). Given $a\geq 0$ and $\ka\in\R$, we define
\begin{equation}\label{eq:Gdef}
\|f\|_{{\G}_a^{\ka,r}} \coloneqq \|e^{a A^{1/2}} f\|_{\hat{W}^{\ka s,r}}.
\end{equation}
We refer to $a$ as the \emph{exponential weight} or \emph{Gevrey index}, $\ka$ as the \emph{Sobolev index}, and $r$ as the \emph{integrability exponent}. If $r<\infty$, then the completion with respect to this norm of functions with compactly supported Fourier transforms in $L^r$  defines a Banach space, as the reader may check. For $0<T<\infty$ and a continuous function $\phi:[0,T]\rightarrow [0,\infty)$, we define
\begin{equation}
\|f\|_{C_T^0{\G}_{\phi}^{\ka,r}} \coloneqq \sup_{0\leq t\leq T} \|f^t\|_{\G_{\phi^t}^{\ka,r}}.
\end{equation}
We write $C_{\infty}^0$ when $\sup_{0\leq t\leq T}$ is replaced by $\sup_{0\leq t<\infty}$. Set
\begin{equation}
C_T^0\G_{\phi}^{\ka,r} \coloneqq \{f\in C([0,T]; \hat{W}^{\ka s,r}(\R^d)) : \|f\|_{C_T^0\G_{\phi}^{\ka,r}} < \infty\}.
\end{equation}
We also allow for $T=\infty$, replacing $[0,T]$ in the preceding line with $[0,\infty)$. Evidently, this defines a Banach space.

To deal with possible issues at low frequencies when $\ga$ is large, we also have need for the Banach spaces
\begin{equation}\label{eq:Xdef}
\X_{a,\ga}^{\ka_1,\ka_2,r_1,r_2} \coloneqq \begin{cases} \G_{a}^{\ka_1,r_1}, & {\ga \leq 1} \\ \G_{a}^{\ka_1,r_1} \cap \G_{a}^{\ka_2,r_2}, & {\ga > 1},\end{cases}
\end{equation}
where
\begin{equation}
\|f\|_{\G_{a}^{\ka_1,r_1} \cap \G_{a}^{\ka_2,r_2}} \coloneqq \|f\|_{\G_{a}^{\ka_1,r_1} } + \|f\|_{\G_{a}^{\ka_2,r_2}},
\end{equation}
and
\begin{equation}
C_T^0\mathsf{X}_{\phi,\ga}^{\ka_1,\ka_2, r_1,r_2} \coloneqq \left\{f\in C\paren*{[0,T]; \hat{W}^{\ka_1 s, r_1}(\R^d)\cap \hat{W}^{\ka_2 s,r_2}(\R^d)} : \|f\|_{C_T^0\mathsf{X}_{\phi,\ga}^{\ka_1,\ka_2,r_1,r_2}}<\infty\right\}.
\end{equation}
The main result of this section is the following proposition.

\begin{prop}\label{prop:lwp}
Let $d\geq 1$, $0<\ga<d+1$, $\frac{1}{2}<s\leq 1$. If $\ga>1$, then also suppose we are given $1\leq q<\frac{d}{\ga-1}$. Given $\al,\be>0$, suppose $W$ is a realization from the set $\Omega_{\al,\be,\nu}$ and that $\beta<\frac{\nu^2}{2}$. Set $\phi^t\coloneqq \al+\beta t$.

There exists $r_0 \geq 1$ depending on $d,\ga,s$, such that the following holds. For any $1\leq r\leq r_0$, there exists $\sigma_0 \in (0,\frac{2s-1}{s})$ depending on $d,\ga,r,s$, such that for any $\sigma \in (\sigma_0,\frac{2s-1}{s})$ with $1-\ga\leq \sigma s$, there exists a constant $C>0$ depending on $d,\ga,s,\sigma,r,q,\beta,\nu$, such that for  $\|\mu^0\|_{\X_{\al,\ga}^{\sigma,0,r,\frac{2q}{q-1}}} \leq R$, there exists a unique solution $\mu \in C_T^0 \X_{\phi,\ga}^{\sigma,0,r,\frac{2q}{q-1}}$ to the Cauchy problem \eqref{eq:cp}, with $T\geq C(|\M|R)^{-\frac{2s}{(2-\sigma)s-1}}$. Moreover,
\begin{equation}
\|\mu\|_{C_T^0\X_{\phi,\ga}^{\sigma,0,r,\frac{2q}{q-1}}} \leq 2\|\mu^0\|_{\X_{\al,\ga}^{\sigma,0,r,\frac{2q}{q-1}}}.
\end{equation}
Additionally, if $\|\mu_j^0\|_{\X_{\al,\ga}^{\sigma,0,r,\frac{2q}{q-1}}}\leq R$, for $j\in\{1,2\}$, then
\begin{equation}
\|\mu_1-\mu_2\|_{C_T^0\X_{\phi,\ga}^{\sigma,0,r,\frac{2q}{q-1}}} \leq 2\|\mu_1^0-\mu_2^0\|_{\X_{\al,\ga}^{\sigma,0,r,\frac{2q}{q-1}}}.
\end{equation}
\end{prop}

\begin{remark}
Compared to statement of \cref{thm:main}, where the Sobolev indices $\sigma_r,\sigma_q$ can be arbitrarily large, \cref{prop:lwp} contains the restriction $\sigma<\frac{2s-1}{s}$ and the second Sobolev index is set to zero. These restrictions are temporary: we only need them to first obtain the existence of a solution. Using a monotonicity argument in the next section, which is in the spirit of persistence of regularity arguments, we then allow for larger values of $\sigma$.
\end{remark}

\begin{remark}
A lower bound for for $r_0$ is explicitly worked out in the proof of \cref{prop:lwp}. See condition \ref{pCon} below and the ensuing analysis.
\end{remark}

\begin{remark}\label{rem:mcon}
The solutions constructed by \cref{prop:lwp} do not a priori conserve mass.\footnote{If we work on $\T^d$ and replace $(1+\Dm^s)$ with $\Dm^s$, then mass is conserved.} To see this, note that by using equation \eqref{eq:cp} and the fundamental theorem of calculus,
\begin{equation}\label{eq:ODEsol}
\frac{d}{dt}\int_{\R^d}\mu^t(x)dx = -\frac{\nu^2}{2}\int_{\R^d}(1+\Dm^s)^2\mu^t(x)dx = -\frac{\nu^2}{2}\int_{\R^d}\mu^t(x)dx,
\end{equation}
where the ultimate equality follows from expanding the square and using the Fourier transform. Solving the ODE \eqref{eq:ODEsol}, we find
\begin{equation}
\int_{\R^d}\mu^t(x)dx = e^{-\frac{\nu^2 t}{2}}\int_{\R^d}\mu^0(x)dx.
\end{equation}
So if $\mu^0$ has zero mass, then $\mu^t$ has zero mass for all times $t$. Otherwise, the magnitude of the mass is exponentially decreasing as $t\rightarrow \infty$. Recalling the mass/diffusion threshold for the PKS equation, this decreasing of the mass of our solutions may provide some intuition why global existence is ultimately possible.
\end{remark}

\subsection{Gevrey embeddings}\label{ssec:lwpGev}
Before proceeding to the contraction mapping step, we record some elementary embeddings satisfied by the spaces $\G_a^{\ka,r}$.

\begin{lemma}\label{lem:Gemb}
If $a'\geq a\geq 0$ and $\ka'\geq \ka$, then
\begin{align}
\|f\|_{\G_{a}^{\ka,r}} \leq e^{a-a'}\|f\|_{\G_{a'}^{\ka',r}}.
\end{align}
If $\ka'\geq \ka$ and $a'>a\geq 0$, then
\begin{equation}
\|f\|_{\G_{a}^{\ka',r}} \leq  \frac{\lceil{\ka'-\ka}\rceil !}{(a'-a)^{\lceil{\ka'-\ka}\rceil}} \|f\|_{\G_{a'}^{\ka,r}}.
\end{equation}
where $\lceil{\cdot}\rceil$ denotes the usual ceiling function.
\end{lemma}
\begin{proof}
First, observe that for any $a'\geq a\geq 0$,
\begin{equation}
\|f\|_{\G_{a}^{\ka,r}} = \|e^{(a-a')A^{1/2}}e^{a' A^{1/2}}f\|_{\hat{W}^{\ka s,r}} \leq e^{a-a'}\|f\|_{\G_{a'}^{\ka,r}},
\end{equation}
since $e^{(a-a')(1+|\xi|^s)} \leq e^{a-a'}$. Also, for any $\ka'\geq \ka$, we trivially have from $\|\cdot\|_{\hat{W}^{\ka s, r}}\leq \|\cdot\|_{\hat{W}^{\ka' s,r}}$ that
\begin{equation}
\|f\|_{\G_{a}^{\ka,r}} \leq \|f\|_{\G_{a}^{\ka',r}}.
\end{equation}
Now for $\ka'\geq \ka$, we have
\begin{align}
\|f\|_{\G_{a}^{\ka',r}} = \|\jp{\nabla}^{\ka' s}e^{a A^{1/2}}f\|_{\hat{L}^r} = \|\jp{\nabla}^{(\ka'-\ka)s}e^{(a-a') A^{1/2}} \jp{\nabla}^{\ka s}e^{a' A^{1/2}}f\|_{\hat{L}^r}.
\end{align}
Observe from the power series for $z\mapsto e^z$ that
\begin{equation}
\jp{\xi}^{(\ka'-\ka) s}e^{(a-a')(1+|\xi|^s)} \leq (1+|\xi|^s)^{\ka'-\ka}e^{(a-a')(1+|\xi|^s)} \leq \frac{\lceil{\ka'-\ka}\rceil !}{(a'-a)^{\lceil{\ka'-\ka}\rceil}},
\end{equation}
where $\lceil{\cdot}\rceil$ is the usual ceiling function. Implicitly, we have used $\|\cdot\|_{\ell^2}\leq \|\cdot\|_{\ell^s}$, since $s\leq 2$. Therefore,
\begin{equation}
\|f\|_{\G_{a}^{\ka',r}} \leq  \frac{\lceil{\ka'-\ka}\rceil !}{(a'-a)^{\lceil{\ka'-\ka}\rceil}} \|f\|_{\G_{a'}^{\ka,r}}.
\end{equation}
\end{proof}

\subsection{Contraction mapping argument}\label{ssec:lwpcm}
Next, we define the map
\begin{equation}
\mu^t\mapsto (\Tc\mu)^t \coloneqq e^{-\frac{t\nu^2}{2}A}\mu^0 - \int_0^t e^{-\frac{(t-\tau)\nu^2}{2}A} B^\tau(\mu^\tau,\mu^\tau)d\tau.
\end{equation}
We check that this map is well-defined on $C_T^0\X_{\phi,\ga}^{\sigma,0,r,\frac{2q}{q-1}}$ for $\phi^t=\al+\be t$, with $\al,\be,\sigma,r,q,\ga>0$ satisfying the conditions in the statement of \cref{prop:lwp}. To this end, we assume here and throughout this subsection that we have a realization of $W$ belonging to $\Omega_{\al,\be,\nu}$.

\begin{lemma}\label{lem:cmlin}
If $\beta<\frac{\nu^2}{2}$, then for any $1\leq r\leq\infty$, $0<s\leq 1$, $\sigma\in\R$, and $\al>0$, it holds that 
\begin{equation}
\|e^{-\frac{t\nu^2}{2}A}\mu^0\|_{\G_{\phi^t}^{\sigma,r}} \leq \|\mu^0\|_{\G_{\al}^{\sigma,r}} \qquad \forall t\geq 0.
\end{equation}
\end{lemma}
\begin{proof}
It is straightforward from the definition of $\phi^t$ that
\begin{align}
\|e^{-\frac{t\nu^2}{2}A}\mu^0\|_{\G_{\phi^t}^{\sigma,r}} = \|e^{t(\be  - \frac{\nu^2}{2}A^{1/2})A^{1/2}}e^{\al A^{1/2}}\mu^0\|_{\hat{W}^{\sigma s,r}}
\end{align}
If $\beta <\frac{\nu^2}{2}$, then since $(1+|\xi|^s)\geq 1$, the right-hand is $\leq \|e^{\al A^{1/2}}\mu^0\|_{\hat{W}^{\sigma s,r}} = \|\mu^0\|_{\G_{\al}^{\sigma,r}}$.
\end{proof}

Next, we observe from the bilinearity of $B$ that 
\begin{equation}
B(\mu_1,\mu_1) - B(\mu_2,\mu_2) = B(\mu_1-\mu_2,\mu_1) + B(\mu_2,\mu_1-\mu_2),
\end{equation}
and therefore
\begin{equation}
\paren*{\Tc(\mu_1) - \Tc(\mu_2)}^t = -\int_0^t e^{-\frac{\nu^2(t-\tau)}{2}A}\paren*{B^\tau(\mu_1^\tau-\mu_2^\tau,\mu_1^\tau) + B^\tau(\mu_2^\tau,\mu_1^\tau-\mu_2^\tau)}d\tau.
\end{equation}

\begin{lemma}\label{lem:cmnl}
Let $d\geq 1$, $0<\ga<d+1$, $\frac{1}{2}<s\leq 1$. If $\ga>1$, also assume that we are given $1\leq q < \frac{d}{\ga-1}$. Additionally, suppose that $\beta < \frac{\nu^2}{2}$.

There exists an $r_0 \in [1,\infty]$, depending on $d,\ga,s$, such that the following holds. For any $1\leq r \leq r_0$, there exists $\sigma_0 \in (0,\frac{2s-1}{s})$ depending on $d,\ga,s,r$, such that for any $\sigma \in (\sigma_0, \frac{2s-1}{s})$ with $1-\ga\leq \sigma s$, there exists a constant $C$ depending only on $d,\ga,r,q,\sigma,s,\beta,\nu$, such that for any $T>0$,
\begin{multline}
\left\|\int_0^t e^{-\frac{\nu^2(t-\tau)}{2}A}B^\tau(\mu_1^\tau,\mu_2^\tau)d\tau\right\|_{C_T^0\G_{\phi}^{\sigma,r}} \leq C|\M|\paren*{T+T^{1-\frac{\sigma s+1}{2s}}}\Bigg(\Big(\|\mu_1\|_{C_T^0\G_{\phi}^{0,r}}\|\mu_2\|_{C_T^0\G_{\phi}^{0,\frac{2q}{q-1}}} \\
+ \|\mu_1\|_{C_T^0\G_{\phi}^{0,\frac{2q}{q-1}}} \|\mu_2\|_{C_T^0\G_{\phi}^{0,\frac{2q}{q-1}}}\Big)\indic_{\ga>1}  + \|\mu_1\|_{C_T^0\G_{\phi}^{\sigma,r}} \|\mu_2\|_{C_T^0\G_{\phi}^{\sigma,r}}\Bigg).
\end{multline}
\end{lemma}

\begin{remark}
We give bounds on the size of the threshold $r_0$ from the statement of \cref{lem:cmnl} during the proof of the lemma. We have omitted them from the statement in order to simplify the presentation. There is an extensive amount of algebra to determine the final conditions on the parameters, but if the reader is not interested in this optimization, they can just consider $r=s=1$, which is straightforward to check.
\end{remark}

\begin{proof}[Proof of \cref{lem:cmnl}]
We make the change of unknown $\mu_j^t\coloneqq e^{-\phi^t A^{1/2}}\jp{\nabla}^{-\sigma s} \rho_j^t$, so that
\begin{equation}
\|\rho_j^t\|_{\hat{L}^r} = \|\mu_j^t\|_{\G_{\phi^t}^{\sigma,r}}.
\end{equation}
By Minkowski's inequality, we see that
\begin{align}
\|e^{\phi^t A^{1/2}} \int_0^{t}e^{-\frac{\nu^2(t-\tau)}{2}A}B^\tau(\mu_1^\tau,\mu_2^\tau) d\tau\|_{\hat{W}^{\sigma s,r}} \leq \int_0^t \|e^{\phi^t A^{1/2} - \frac{\nu^2(t-\tau)}{2}A}B^\tau(\mu_1^\tau,\mu_2^\tau)\|_{\hat{W}^{\sigma s,r}}d\tau,
\end{align}
and by definition of the $\hat{W}^{\sigma s,r}$ norm, the preceding right-hand side equals
\begin{multline}
\int_0^t \Bigg(\int_{\R^d} e^{r\phi^t(1+|\xi|^s) - \nu^2(t-\tau)(1+|\xi|^s)^2}\jp{\xi}^{r\sigma s} \left|e^{-\nu W^\tau(1+|\xi|^s)}\int_{\R^d}\frac{(\xi\cdot\M\eta)\hat{\g}(\eta)}{\jp{\xi-\eta}^{\sigma s}\jp{\eta}^{\sigma s}} \right.\\
\left. e^{-\phi^\tau(2+|\xi-\eta|^s +|\eta|^s)}e^{\nu W^\tau(2+|\xi-\eta|^s+|\eta|^s)}\hat{\rho}_1^\tau(\xi-\eta)\hat{\rho}_2^\tau(\eta) d\eta \right|^r\Bigg)^{1/r}d\tau.
\end{multline}
Using $\phi^t-\phi^\tau = \be(t-\tau)$, the preceding expression is controlled by
\begin{multline}
\int_0^t\Bigg(\int_{\R^d}e^{r(t-\tau)(1+|\xi|^s)(\beta-\frac{\nu^2}{2}(1+|\xi|^s))}\jp{\xi}^{r\sigma s}  \left|\int_{\R^d} e^{(\phi^\tau-\nu W^{\tau})(|\xi|^s-|\xi-\eta|^s-|\eta|^s-1)} \right.\\
\left.\frac{|\xi\cdot\M\eta| |\hat{\g}(\eta)|}{\jp{\xi-\eta}^{\sigma s}\jp{\eta}^{\sigma s}}\left|\hat{\rho}_1^\tau(\xi-\eta)\right|\left|\hat{\rho}_2^\tau(\eta)\right| d\eta \right|^r d\xi\Bigg)^{1/r}d\tau.
\end{multline}
Since $0<s\leq 1$, and therefore by $\|\cdot\|_{\ell^1}\leq \|\cdot\|_{\ell^s}$, and $\phi^\tau-\nu W^\tau\geq 0$ for all $0\leq\tau\leq t$ by assumption, it follows from the triangle inequality that
\begin{equation}
e^{(\phi^\tau-\nu W^{\tau})(|\xi|^s-|\xi-\eta|^s-|\eta|^s-1)} \leq 1.
\end{equation}
Now set $\d \coloneqq \min\{\frac{1}{2}(\frac{\nu^2}{2}-\be), \frac{\nu^2}{4}\}$. By assumption that $\beta < \frac{\nu^2}{2}$, we have
\begin{equation}
e^{r(t-\tau)(1+|\xi|^s)(\beta-\frac{\nu^2}{2}(1+|\xi|^s))} \leq e^{-r\d(t-\tau)(1+|\xi|^s)^2}.
\end{equation}
With these observations, we reduce to estimate the expression
\begin{equation}\label{eq:LWPnte}
\int_0^t \Bigg(\int_{\R^d} e^{-r\d(t-\tau)(1+|\xi|^s)^2}\jp{\xi}^{r\sigma s}\left|\int_{\R^d}\frac{|\xi\cdot\M\eta| |\hat{\g}(\eta)|}{\jp{\xi-\eta}^{\sigma s}\jp{\eta}^{\sigma s}}\left|\hat{\rho}_1^\tau(\xi-\eta)\right|\left|\hat{\rho}_2^\tau(\eta)\right| d\eta \right|^r d\xi\Bigg)^{1/r} d\tau
\end{equation}

To deal with the inhomogeneity of $(1+|\xi|^s)^2$, we split the integral with respect to $\xi$ into the low-frequency piece $|\xi|\leq 1$ and the high-frequency piece $|\xi|>1$. At low frequency, we can crudely estimate everything directly to find
\begin{multline}
\Bigg(\int_{|\xi|\leq 1} e^{-r\d(t-\tau)(1+|\xi|^s)^2}\left|\int_{\R^d}\frac{|\xi\cdot\M\eta| |\hat{\g}(\eta)|}{\jp{\xi-\eta}^{\sigma s}\jp{\eta}^{\sigma s}}\left|\hat{\rho}_1^\tau(\xi-\eta)\right|\left|\hat{\rho}_2^\tau(\eta)\right| d\eta \right|^r d\xi\Bigg)^{1/r} \\
\lesssim_\ga |\M|\Bigg(\int_{|\xi|\leq 1}\left|\int_{\R^d}\jp{\xi-\eta}^{-\sigma s}|\eta|^{1-\ga}\jp{\eta}^{-\sigma s}\left| \hat{\rho}_1^\tau(\xi-\eta)\right|\left|\hat{\rho}_2^\tau(\eta)\right| d\eta \right|^r d\xi \Bigg)^{1/r}.
\end{multline}
Above, we have used our assumption that $|\eta\hat{\g}(\eta)|\lesssim_\ga \eta^{1-\ga}$. If $1-\ga\geq 0$, then $|\eta|^{1-\ga}\jp{\eta}^{-\sigma s} \leq \jp{\eta}^{1-\ga -\sigma s}$. If $1-\ga<0$, then we have to be careful about singularities at low frequency. More precisely, by H\"older's inequality, we can control the $L_{\xi}^r$ norm by the $L_{\xi}^\infty$ norm. For $q(1-\ga)>-d$, H\"older's inequality gives
\begin{align}
&\left|\int_{|\eta|\leq 1}\jp{\xi-\eta}^{-\sigma s}|\eta|^{1-\ga}\jp{\eta}^{-\sigma s} |\hat{\rho}_1^\tau(\xi-\eta)| |\hat{\rho}_2^\tau(\eta)|d\eta\right| \nn\\
&\leq \||\cdot|^{1-\ga}1_{B(0,1)} \|_{L^q} \|\jp{\xi-\cdot}^{-\sigma s}\jp{\cdot}^{-\sigma s}\hat{\rho}_1^\tau(\xi-\cdot) \hat{\rho}_2^\tau1_{B(0,1)}\|_{L^{\frac{q}{q-1}}} \nn\\
&\lesssim_{d,\ga,q} \|\hat{\rho}_1^\tau\|_{\hat{W}^{-\sigma s,\frac{2q}{q-1}}} \|\hat{\rho}_2^\tau\|_{\hat{W}^{-\sigma s,\frac{2q}{q-1}}} \nn\\
&\lesssim \|e^{\phi^\tau A^{1/2}}\mu_1^\tau\|_{\hat{W}^{0,\frac{2q}{q-1}}}\|e^{\phi^\tau A^{1/2}}\mu_2^\tau\|_{\hat{W}^{0,\frac{2q}{q-1}}}. \label{eq:etalwplf}
\end{align}
The remaining expression
\begin{align}
&\Bigg(\int_{|\xi|\leq 1} \left|\int_{|\eta|\geq 1}\jp{\xi-\eta}^{-\sigma s}|\eta|^{1-\ga}\jp{\eta}^{-\sigma s}\left|\hat{\rho}_1^\tau(\xi-\eta)\right|\left|\hat{\rho}_2^\tau(\eta)\right| d\eta \right|^r d\xi\Bigg)^{1/r} 
\end{align}
is handled by estimate \eqref{eq:lwphf} below.

At high frequency, we trivially have
\begin{equation}
e^{-r\d(t-\tau)(1+|\xi|^s)^2}\leq e^{-r\d(t-\tau)|\xi|^{2s}}.
\end{equation}
Writing $|\xi| = (t-\tau)^{-\frac{1}{2s}}(t-\tau)^{\frac{1}{2s}}|\xi|$, it follows from the power series for $z\mapsto e^{z}$ that
\begin{equation}
e^{-r\d(t-\tau)(1+|\xi|^s)^2}\jp{\xi}^{r\sigma s}|\xi|^{r} \lesssim_\d (t-\tau)^{-\frac{r(\sigma s+1)}{2s}}.
\end{equation}
Hence,
\begin{multline}
\Bigg(\int_{|\xi|> 1} e^{-r\d(t-\tau)(1+|\xi|^s)^2} \jp{\xi}^{r\sigma s}\left|\int_{\R^d}\frac{|\xi\cdot\M\eta| |\hat{\g}(\eta)|}{\jp{\xi-\eta}^{\sigma s}\jp{\eta}^{\sigma s}}\left|\hat{\rho}_1^\tau(\xi-\eta)\right|\left|\hat{\rho}_2^\tau(\eta)\right| d\eta \right|^r d\xi\Bigg)^{1/r} \\
\lesssim_{\ga,\d,\sigma,s} |\M|(t-\tau)^{-\frac{(\sigma s+1)}{2s}} \Bigg(\int_{|\xi|>1} \left|\int_{\R^d} \jp{\xi-\eta}^{-\sigma s}|\eta|^{1-\ga}\jp{\eta}^{-\sigma s}\left|\hat{\rho}_1^\tau(\xi-\eta)\right|\left|\hat{\rho}_2^\tau(\eta)\right| d\eta \right|^r d\xi\Bigg)^{1/r}.
\end{multline}
As before, we have to be careful about singularities in $\eta$ at low frequency if $1-\ga< 0$. Observe from Young's inequality that
\begin{align}
&\Bigg(\int_{|\xi|>1} \left|\int_{|\eta|\leq 1} \jp{\xi-\eta}^{-\sigma s}|\eta|^{1-\ga}\jp{\eta}^{-\sigma s}\left|\hat{\rho}_1^\tau(\xi-\eta)\right|\left|\hat{\rho}_2^\tau(\eta)\right| d\eta \right|^r d\xi\Bigg)^{1/r} \nn\\
&\leq \|\jp{\cdot}^{-\sigma s}\hat{\rho}_1^\tau\|_{L^r} \||\cdot|^{1-\ga}\hat{\rho}_2^\tau 1_{B(0,1)}\|_{L^1} \nn\\
&\lesssim_{d,\ga,\sigma,s,q} \|e^{\phi^\tau A^{1/2}} \mu_1^\tau\|_{\hat{W}^{0,r}} \|e^{\phi^\tau A^{1/2}}\mu_2^\tau\|_{\hat{W}^{0,\frac{2q}{q-1}}}, \label{eq:etalwphf}
\end{align}
for any $q<\frac{d}{\ga-1}$. For $\eta$ at high frequency, we have by Young's inequality followed by \cref{lem:Sob} that any $1\leq p\leq r$,
\begin{align}
&\Bigg(\int_{\R^d} \left|\int_{|\eta|\geq 1} \jp{\xi-\eta}^{-\sigma s}\jp{\eta}^{1-\ga-\sigma s}\left|\hat{\rho}_1^\tau(\xi-\eta)\right|\left|\hat{\rho}_2^\tau(\eta)\right| d\eta \right|^r d\xi\Bigg)^{1/r} \nn\\
&\leq \|\jp{\cdot}^{-\sigma s}\hat{\rho}_1^\tau\|_{L^p} \|\jp{\cdot}^{1-\ga-\sigma s} \hat{\rho}_2^\tau\|_{L^{\frac{rp}{(r+1)p-r}}} \nn\\
&\lesssim_{\sigma,s,d,\ga,r,p}  \|\rho_1^\tau\|_{\hat{W}^{-\sigma s, 1}} \|\rho_2^\tau\|_{\hat{W}^{1-\ga-\sigma s, 1}}\indic_{r=1} + \|\rho_1^\tau\|_{\hat{W}^{(\frac{d(r-1)}{r}-\sigma s)+, r}} \|\rho_2^\tau\|_{\hat{W}^{1-\ga-\sigma s, r}}\indic_{\substack{p=1 \\ r>1}} \nn\\
&\ph + \|\rho_1^\tau\|_{\hat{W}^{-\sigma s,r}} \|\rho_2^\tau\|_{\hat{W}^{(1-\ga-\sigma s + \frac{d(r-1)}{r})+, r}}\indic_{\substack{ p=r \\r>1}}\nn\\
&\ph +\|\rho_1^\tau\|_{\hat{W}^{(\frac{d(r-p)}{rp}-\sigma s)+,r}} \|\rho_2^\tau\|_{\hat{W}^{(1-\ga-\sigma s+\frac{d(p-1)}{p})+, r}}\indic_{\substack{1<p<r \\ r>1}} \nn\\
&= \|e^{\phi^\tau A^{1/2}}\mu_1^\tau\|_{\hat{W}^{0,1}}\|e^{\phi^\tau A^{1/2}}\mu_2^\tau\|_{\hat{W}^{1-\ga,1}}\indic_{r=1} + \|e^{\phi^\tau A^{1/2}}\mu_1^\tau\|_{\hat{W}^{\frac{d(r-1)}{r}+,r}}\|e^{\phi^\tau A^{1/2}}\mu_2^\tau\|_{\hat{W}^{1-\ga,r}}\indic_{\substack{p=1 \\ r>1}} \nn\\
&\ph +  \|e^{\phi^\tau A^{1/2}}\mu_1^\tau\|_{\hat{W}^{0,r}}\|e^{ \phi^\tau A^{1/2}}\mu_2^\tau\|_{\hat{W}^{(1-\ga + \frac{d(r-1)}{r})+,r}}\indic_{\substack{ p=r \\r>1}}\nn \\
&\ph + \|e^{\phi^\tau A^{1/2}}\mu_1^\tau\|_{\hat{W}^{(\frac{d(r-p)}{rp})+,r}} \| e^{\phi^\tau A^{1/2}}\mu_2^\tau\|_{\hat{W}^{(1-\ga+\frac{d(p-1)}{p})+, r}}\indic_{\substack{1<p<r \\ r>1}} . \label{eq:lwphf}
\end{align}
In order to obtain estimates that close, we need the top Sobolev index appearing in \eqref{eq:lwphf} to be $\leq \sigma s$. This leads us to the following conditions:
\begin{equation}
\begin{cases}
1-\ga \leq \sigma s, & {r=1} \\
\frac{d(r-1)}{r}<\sigma s \ \text{and} \  1-\ga \leq \sigma s, & {p=1 \ \text{and} \ r>1} \\
1-\ga+\frac{d(r-1)}{r}<\sigma s, & {p=r \ \text{and} \ r>1} \\
\frac{d(r-p)}{rp}<\sigma s \ \text{and} \ 1-\ga+\frac{d(p-1)}{p} < \sigma s, & {1<p<r \ \text{and} \ r>1}.
\end{cases}
\end{equation}
Putting together the estimates \eqref{eq:etalwphf} and \eqref{eq:lwphf}, we have shown that
\begin{multline}\label{eq:lwphff}
\Bigg(\int_{|\xi|>1} \left|\int_{\R^d} \jp{\xi-\eta}^{-\sigma s}|\eta|^{1-\ga}\jp{\eta}^{-\sigma s}\left|\hat{\rho}_1^\tau(\xi-\eta)\right|\left|\hat{\rho}_2^\tau(\eta)\right| d\eta \right|^r d\xi\Bigg)^{1/r}\\
\lesssim_{d,s,\sigma,\ga,r,q} \|e^{\phi^\tau A^{1/2}} \mu_1^\tau\|_{\hat{W}^{0,r}}\|e^{\phi^\tau A^{1/2}}\mu_2^\tau\|_{\hat{W}^{0,\frac{2q}{q-1}}}\indic_{\ga>1}  \\
+\|e^{\phi^\tau A^{1/2}} \mu_1^\tau\|_{\hat{W}^{\sigma s,r}} \|e^{\phi^\tau A^{1/2}}\mu_2^\tau\|_{\hat{W}^{\sigma s,r}}.
\end{multline}

Combining the estimates \eqref{eq:etalwplf} and \eqref{eq:lwphff}, we have shown that \eqref{eq:LWPnte} is $\lesssim_{d,\ga,s,\sigma,r,q} $
\begin{multline}
|\M|\int_0^t \paren*{1+ C_\d (t-\tau)^{-\frac{(\sigma s+1)}{2s}}} \Bigg(\Big(\|e^{\phi^\tau A^{1/2}} \mu_1^\tau\|_{\hat{W}^{0,r}} \|e^{\phi^\tau A^{1/2}}\mu_2^\tau\|_{\hat{W}^{0,\frac{2q}{q-1}}}  \\
+ \|e^{\phi^\tau A^{1/2}} \mu_1^\tau\|_{\hat{W}^{0,\frac{2q}{q-1}}}\|e^{\phi^\tau A^{1/2}} \mu_1^\tau\|_{\hat{W}^{0,\frac{2q}{q-1}}} \Big)\indic_{\ga>1} +\|e^{\phi^\tau A^{1/2}} \mu_1^\tau\|_{\hat{W}^{\sigma s,r}} \|e^{\phi^\tau A^{1/2}}\mu_2^\tau\|_{\hat{W}^{\sigma s,r}}\Bigg)d\tau \\
\lesssim_{\sigma,s} |\M|\paren*{t+C_\d t^{1-\frac{(\sigma s+1)}{2s}}}\Bigg(\Big(\|\mu_1\|_{C_t^0\G_{\phi}^{0,r}}\|\mu_2\|_{C_t^0\G_{\phi}^{0,\frac{2q}{q-1}}} + \|\mu_1\|_{C_t^0\G_{\phi}^{0,\frac{2q}{q-1}}} \|\mu_2\|_{C_t^0\G_{\phi}^{0,\frac{2q}{q-1}}}\Big)\indic_{\ga>1}  \\
+ \|\mu_1\|_{C_t^0\G_{\phi}^{\sigma,r}} \|\mu_2\|_{C_t^0\G_{\phi}^{\sigma,r}}\Bigg),
\end{multline}
assuming that $\frac{(\sigma s+1)}{2s} <1$.

In order to complete the proof of the lemma, it is important to list all the conditions we imposed on the parameters $d,\ga,\sigma,s,r$ during the course of the above analysis:
\begin{enumerate}[(LWP1)]
\item\label{sCon}
$0<s\leq 1$;
\item\label{pCon}
	\begin{enumerate}
	\item\label{pCona}
	$r=1$ and $1-\ga\leq \sigma s$,
	\item\label{pConb}
	or $r>1$ and $\frac{d(r-1)}{r}<\sigma s$ and $1-\ga\leq \sigma s$,
	\item\label{pConc}
	or $r>1$ and $1-\ga + \frac{d(r-1)}{r}<\sigma s$,
	\item\label{pCond}
	or $r>1$ and $\exists p\in (1,r)$ such that $\frac{d(r-p)}{rp}<\sigma s$ and $1-\ga +\frac{d(p-1)}{p}<\sigma s$;
	\end{enumerate}
\item\label{ssCon}
$\frac{(\sigma s+1)}{2s}<1$
\end{enumerate}
Condition \ref{ssCon} means $\sigma < \frac{2s-1}{s}$ and since we require $\sigma>0$, we need $s>\frac{1}{2}$. Given any value of $0<\ga< d+1$, \ref{pCona} can be satisfied by choosing $r=1$ and $\sigma s\geq \min\{0,1-\ga\}$. More generally, we can satisfy all three conditions by arguing as follows. Given $0<\ga < d+1$ and $\frac{1}{2}<s\leq 1$, condition \ref{ssCon} implies that for any choice $r\geq 1$,
\begin{equation}
\sigma s < 2s-1.
\end{equation}
According to \ref{pConb}, it is possible to find such a $\sigma\geq \frac{1-\ga}{s}$ and $r>1$ if and only if
\begin{equation}
\frac{d(r-1)}{r} < 2s-1 \ \text{and} \ 1-\ga < 2s-1 \quad \Longleftrightarrow \quad r< \frac{d}{d+1-2s} \ \text{and} \ \frac{2-\ga}{2} < s .
\end{equation}
According to \ref{pConc}, it is possible to find such a $\sigma>0$ and $r>1$ if and only if
\begin{equation}
1-\ga + \frac{d(r-1)}{r} < 2s-1 \quad \Longleftrightarrow \quad \begin{cases} r<\frac{d}{d+2-\ga-2s}, & {\ga+2s\leq d+2} \\ r\leq \infty, &{\ga+2s> d+2}.\end{cases}
\end{equation}
According to \ref{pCond}, it is possible to find such a $\sigma >0$ and $r>1$ if and only if for such choice of $r$, there exists $p\in (1,r)$ such that
\begin{equation}
\frac{d(r-p)}{rp} < 2s-1 \quad \text{and} \quad 1- \ga + \frac{d(p-1)}{p} < 2s-1.
\end{equation} 
Since the preceding constraint is equivalent to
\begin{equation}
\frac{d}{p} < \frac{d}{r}+2s-1 \quad \text{and} \quad d+2-\ga-2s < \frac{d}{p},
\end{equation}
such a $p$ exists if and only if
\begin{equation}
d+2-\ga-2s < \frac{d}{r}+2s-1 \quad \Longleftrightarrow \quad \begin{cases} r<\frac{d}{d+3-\ga-4s}, & \ga+4s\leq d+3 \\ r\leq \infty , & \ga+4s > d+3. \end{cases}
\end{equation}
With this case analysis, the proof of \cref{lem:cmnl} is now complete.
\end{proof}

\begin{lemma}\label{lem:cmnlq}
Let $d\geq 1$, $1<\ga<d+1$. Suppose that $\ga,r,s,\sigma$ satisfy the constraints of \cref{lem:cmnl} and that
\begin{equation}
d+1-\ga<\sigma s + \frac{d}{r}.
\end{equation}
Then for any $1\leq q <\frac{d}{\ga-1}$, there exists a constant $C>0$ depending on $d,\ga,r,q,s,\sigma,\beta,\nu$, such that for any $T>0$,
\begin{multline}
\left\|\int_0^t e^{-\frac{\nu^2(t-\tau)}{2}A}B^\tau(\mu_1^\tau,\mu_2^\tau)d\tau\right\|_{C_T^0\G_{\phi}^{0,\frac{2q}{q-1}}} \leq C|\M|\paren*{T+T^{1-\frac{\sigma s+1}{2s}}}\Bigg(\|\mu_1\|_{C_{T}^0\G_{\phi}^{0,\frac{2q}{q-1}}}\|\mu_2\|_{C_{T}^0\G_{\phi}^{0,\frac{2q}{q-1}}} \\
+ \|\mu_1\|_{C_{T}^0\G_{\phi}^{\sigma , r}} \|\mu_2\|_{C_{T}^0\G_{\phi}^{0,\frac{2q}{q-1}}}\Bigg).
\end{multline}
\end{lemma}
\begin{proof}
Set $q'\coloneqq\frac{q}{q-1}$. The proof follows the exact same lines of \cref{lem:cmnl} with $r$ replaced by $2q'$. Using the estimates \eqref{eq:etalwplf}, \eqref{eq:etalwphf}, \eqref{eq:lwphf}, we find that
\begin{multline}\label{eq:lwpqrhs}
\left\|\int_0^t e^{-\frac{\nu^2(t-\tau)}{2}A}B^\tau(\mu_1^\tau,\mu_2^\tau)d\tau\right\|_{\G_{\phi^t}^{0,2q'}}  \\
\lesssim_{d,s,q,\ga,\beta,\nu,p} |\M|\int_0^t \paren*{1+ (t-\tau)^{-\frac{\sigma s+1}{2s}}}\Bigg(\|e^{\phi^\tau A^{1/2}}\mu_1^\tau\|_{\hat{W}^{0,2q'}}\|e^{\phi^\tau A^{1/2}}\mu_2^\tau\|_{\hat{W}^{0,2q'}}\\
+\|e^{\phi^\tau A^{1/2}}\mu_1^\tau\|_{\hat{W}^{0,p}} \|e^{\phi^\tau A^{1/2}} \mu_2^\tau\|_{\hat{W}^{1-\ga,\frac{2q'p}{(2q'+1)p-2q'}}}\Bigg)d\tau
\end{multline}
for any choice $1\leq p\leq 2q'$. We want all the norms appearing in the right-hand side to be controlled by $\hat{W}^{0,2q'}$ and $\hat{W}^{\sigma s, r}$. Using \cref{lem:Sob}, we see that we need to choose $p\leq r$ so that
\begin{equation}
\frac{d}{p} - \frac{d}{r} < \sigma s \Longleftrightarrow {\sigma s + \frac{d}{r}} > \frac{d}{p}.
\end{equation}
For any choice of $p$, we have
\begin{equation}
\frac{2q'p}{(2q'+1)p-2q'} = 2q'\paren*{\frac{p}{p + 2q'(p-1)}} \leq 2q',
\end{equation}
since $p\geq 1$. In order for
\begin{equation}
\|e^{\phi^\tau A^{1/2}} \mu_2^\tau\|_{\hat{W}^{1-\ga,\frac{2q'p}{(2q'+1)p-2q'}}} \lesssim \|e^{\phi^\tau A^{1/2}} \mu_2^\tau\|_{\hat{W}^{0,2q'}},
\end{equation}
another use of \cref{lem:Sob} tells us that we need
\begin{equation}
1-\ga + d\paren*{\frac{(2q' + 1)p - 2q'}{2q' p} - \frac{1}{2q'}} = 1-\ga +\frac{d(p-1)}{p} < 0 \Longleftrightarrow d+1-\ga < \frac{d}{p}.
\end{equation}
Since we also needed $\sigma s+\frac{d}{r}>\frac{d}{p}$, the existence of a $p$ satisfying both the upper and lower bounds is true if and only if
\begin{equation}
d+1-\ga < \sigma s + \frac{d}{r} \Longleftrightarrow \begin{cases} r < \frac{d}{d+1-\ga-\sigma s}, & \ga+\sigma s \leq d+1 \\ r\leq \infty, & \ga+\sigma s>d+1.\end{cases}
\end{equation}
Note that if $\ga\geq 1$, then for any $\sigma s>0$, we have $\frac{d}{d+1-\ga-\sigma s} > 1$.

Under the constraints of the preceding paragraph, the right-hand side of \eqref{eq:lwpqrhs} is $\lesssim_{d,\ga,r,q,\sigma, s}$
\begin{multline}
|\M|\int_0^t \paren*{1+(t-\tau)^{-\frac{\sigma s+1}{2s}}}\Bigg(\|e^{\phi^\tau A^{1/2}}\mu_1^\tau\|_{\hat{W}^{0,2q'}}\|e^{\phi^\tau A^{1/2}}\mu_2^\tau\|_{\hat{W}^{0,2q'}} \\
+ \|e^{\phi^\tau A^{1/2}}\mu_1^\tau\|_{\hat{W}^{\sigma s, r}} \|e^{\phi^\tau A^{1/2}}\mu_2^\tau\|_{\hat{W}^{0, 2q'}} \Bigg)d\tau.
\end{multline}
Taking the supremum over $\tau\in [0,T]$ in the right-hand side and using fundamental theorem of calculus leads to the desired conclusion.
\end{proof}

An immediate corollary of \Cref{lem:cmnl,lem:cmnlq} is the following estimate for the Duhamel term.

\begin{cor}\label{cor:cmnl}
Under the assumptions of \Cref{lem:cmnl,lem:cmnlq}, there exists a constant $C>0$ depending on $d,\ga,r,q,\sigma, s, \beta,\nu$, such that for any $T>0$ and $\mu_1,\mu_2 \in C_T^0\X_{\phi,\ga}^{\sigma,0,r,\frac{2q}{q-1}}$,
\begin{multline}
\left\|\int_0^t e^{-\frac{\nu^2(t-\tau)}{2}A}B^\tau(\mu_1^\tau,\mu_2^\tau)d\tau\right\|_{C_T^0\X_{\phi,\ga}^{\sigma,0,r,\frac{2q}{q-1}}} \\
\leq C|\M|\paren*{T+T^{1-\frac{\sigma s+1}{2s}}}\Bigg(\|\mu_1\|_{C_T^0\X_{\phi,\ga}^{\sigma,0,r,\frac{2q}{q-1}}}\|\mu_2\|_{C_T^0\X_{\phi,\ga}^{\sigma,0,r,\frac{2q}{q-1}}}\Bigg)
\end{multline}
\end{cor}

\begin{proof}[Proof of \cref{prop:lwp}]
Putting together the estimates of \cref{lem:cmlin} and \cref{cor:cmnl}, we have shown that there exists a constant $C>0$ depending on $d,\ga,r,q,\sigma,s,\beta,\nu$, such that
\begin{equation}\label{eq:Tmu}
\|\Tc(\mu)\|_{C_T^0\X_{\phi,\ga}^{\sigma,0,r,\frac{2q}{q-1}}} \leq \|\mu^0\|_{\X_{\al,\ga}^{\sigma,0,r,\frac{2q}{q-1}}} + C|\M|\paren*{T+T^{1-\frac{(\sigma s+1)}{2s}}}\|\mu\|_{C_T^0\X_{\phi,\ga}^{\sigma,0,r,\frac{2q}{q-1}}}^2
\end{equation}
and
\begin{multline}\label{eq:Tmu12}
\|\Tc(\mu_1)-\Tc(\mu_2)\|_{C_T^0\X_{\phi,\ga}^{\sigma,0,r,\frac{2q}{q-1}}} \leq C|\M|\paren*{T+T^{1-\frac{(\sigma s+1)}{2s}}}\|\mu_1-\mu_2\|_{C_T^0\X_{\phi,\ga}^{\sigma,0,r,\frac{2q}{q-1}}}\\
\times \paren*{\|\mu_1\|_{C_T^0\X_{\phi,\ga}^{\sigma,0,r,\frac{2q}{q-1}}}+\|\mu_2\|_{C_T^0\X_{\phi,\ga}^{\sigma,0,r,\frac{2q}{q-1}}}}.
\end{multline}
We now want to show that for any appropriate choice of $T$, the map $\Tc$ is a contraction on the closed ball $B_{R}(0)$ of radius $R>2\|\mu^0\|_{\X_{\al,\ga}^{\sigma,0,r,\frac{2q}{q-1}}}$ centered at the origin in the space $C_T^0\X_{\phi,\ga}^{\sigma,0,r,\frac{2q}{q-1}}$. Indeed, from the estimates \eqref{eq:Tmu} and \eqref{eq:Tmu12}, we see that if
\begin{equation}
C|\M|R\paren*{T+T^{1-\frac{\sigma s+1}{2s}}} \leq \frac{1}{2},
\end{equation}
then $\Tc$ is a contraction on $B_R(0)$. So by the contraction mapping theorem, there exists a unique fixed point $\mu=\Tc(\mu) \in C_T^0\X_{\phi,\ga}^{\sigma,0,r,\frac{2q}{q-1}}$. We note that $T\geq C'(|\M|R)^{-\frac{2s}{2s-\sigma s-1}}$, where $C'>0$ is a possibly different constant than $C$ but depending on the same parameters.

The preceding result shows local existence and uniqueness of solutions to the Cauchy problem \eqref{eq:cp}. To complete the proof of \cref{prop:lwp}, we now prove continuous dependence on the initial data. For $j=1,2$, let $\mu_j$ be a solution in $C_{T_j}^{0}\X_{\phi,\ga}^{\sigma,0,r,\frac{2q}{q-1}}$ to \eqref{eq:cp} with initial datum $\mu_j^0$, such that $\|\mu_j^0\|_{\X_{\al,\ga}^{\sigma,0,r,\frac{2q}{q-1}}}\leq R$. Then there exists a $T \gtrsim_{d,\ga,r,q,\sigma,s,\beta,\nu} (|\M|R)^{-\frac{2s}{2s-\sigma s-1}}$ such that $\mu_1,\mu_2$ are defined on $[0,T]$. From the mild formulation \eqref{eq:mild}, the triangle inequality, \cref{lem:cmlin}, and \cref{cor:cmnl}, we see that
\begin{multline}\label{eq:cdre}
\|\mu_1-\mu_2\|_{C_T^0\X_{\phi,\ga}^{\sigma,0,r,\frac{2q}{q-1}}} \leq \|\mu_1^0-\mu_2^0\|_{\X_{\al,\ga}^{\sigma,0,r,\frac{2q}{q-1}}} \\
+ C|\M|\paren*{T+T^{1-\frac{\sigma s+1}{2s}}}\|\mu_1-\mu_2\|_{C_T^0\X_{\phi,\ga}^{\sigma,0,r,\frac{2q}{q-1}}}\paren*{\|\mu_1\|_{C_T^0\X_{\phi,\ga}^{\sigma,0,r,\frac{2q}{q-1}}} + \|\mu_2\|_{C_T^0\X_{\phi,\ga}^{\sigma,0,r,\frac{2q}{q-1}}}}.
\end{multline}
Taking $T$ smaller if necessary while still preserving $T\gtrsim_{d,\ga,r,q,s,\sigma,\beta,\nu} (|\M|R)^{-\frac{2s}{2s-\sigma s-1}}$, we may assume that $2C|\M|(T+T^{1-\frac{\sigma s+1}{2}})R \leq\frac{1}{4}$. Bounding each $\|\mu_j\|_{C_T^0\X_{\phi,\ga}^{\sigma,0,r,\frac{2q}{q-1}}}$ by $R$ in the last factor, it then follows from \eqref{eq:cdre} that
\begin{equation}
\|\mu_1-\mu_2\|_{C_T^0\X_{\phi,\ga}^{\sigma,0,r,\frac{2q}{q-1}}}\leq 2\|\mu_1^0-\mu_2^0\|_{\X_{\al,\ga}^{\sigma,0,r,\frac{2q}{q-1}}}.
\end{equation}
With this last estimate, the proof of \cref{prop:lwp} is complete.
\end{proof}

\section{Global existence}\label{sec:glob}
We now show that with quantifiable high probability, there exists a global solution $\mu \in C_{\infty}^0\X_{\phi,\ga}^{\sigma_r,\sigma_q,r,\frac{2q}{q-1}}$ to the Cauchy problem \eqref{eq:cp}, provided $\sigma_r,\sigma_q,r,q$ are appropriately chosen. Moreover, the function
\begin{equation}
t \mapsto \|\mu^t\|_{\X_{\phi^t,\ga}^{\sigma_r,\sigma_q,r,q}}
\end{equation}
is strictly decreasing on $[0,\infty)$, provided that $\|\mu^0\|_{\X_{\al,\ga}^{\sigma_r,\sigma_q,r,q}}$ is sufficiently small. This then proves \cref{thm:main}. 

\subsection{Monotonicity of Gevrey norm}\label{ssec:globmon}
The goal of this subsection is to show the following. Suppose we have a solution $\mu \in C_T^0\X_{\phi,\ga}^{\ka_r,\ka_q, r,\frac{2q}{q-1}}$ to \eqref{eq:cp}, where $\phi^t\coloneqq \al+\beta t$, such that $\mu$ also belongs to $C_T^0\X_{\phi,\ga}^{\ka_r',\ka_q',r,\frac{2q}{q-1}}$, for sufficiently larger $\ka_r'>\ka_r$ and $\ka_q'>\ka_q$, and such that $\|\mu^0\|_{\X_{\al,\ga}^{\ka_r,\ka_q,r,\frac{2q}{q-1}}}$ is sufficiently small depending on $d,\ga, r,q,\ka_r,\ka_q,s,\beta,\nu,|\M|$. If $\ka_r,\ka_q$ are sufficiently large depending on $d,s,\ga$, then the quantity $\|\mu^t\|_{\X_{\phi^t,\ga}^{\ka_r,\ka_q, r,\frac{2q}{q-1}}}$ must be strictly decreasing on the interval $[0,T]$. In other words, the Gevrey norm of $\mu^t$ is strictly decreasing on an interval, provided that we know a Gevrey norm with higher Sobolev index (but the same Gevrey index) remains finite on the same interval.

\begin{prop}\label{prop:mon}
Let $d\geq 1$, $0<\ga<d+1$, $1\leq r\leq \infty$, $\frac{1}{2}<s\leq 1$. If $\ga>1$, then also suppose we are given $1<q<\frac{d}{\ga-1}$. Given $\al,\be>0$, set $\phi^t\coloneqq \al+\be t$. Assume that $W$ is a realization from $\Omega_{\al,\be,\nu}$.

If $\ga\leq 1$, then there is a threshold $\ka_{0,r}\in\R$ depending on $r,d,s,\ga$, such that for any $\ka_r>\ka_{0,r}$, the following holds. There is a constant $C_r>0$, depending only on $d,\ga,r,s,\ka_r$, such that if $\mu \in C_T^0\G_{\phi}^{\ka_r+\frac{2}{r},r}$ is a solution to \eqref{eq:cp}, for some $T>0$, satisfying
\begin{equation}\label{eq:monidr}
\|\mu^0\|_{\G_{\al}^{\ka_r,r}} < \frac{\nu^2-2\beta}{C_r|\M|},
\end{equation}
then
\begin{equation}\label{eq:monr}
\|\mu^t\|_{\G_{\phi^t}^{\ka_r,r}} < \|\mu^{t'}\|_{\G_{\phi^{t'}}^{\ka_r,r}} \qquad \forall 0\leq t' < t \leq T.
\end{equation}

If $\ga>1$, then there is a threshold $\ka_{0,q}\in\R$ depending on $q,d,s,\ga$, such that for any $\ka_q>\ka_{0,q}$, the following holds. There is a constant $C_q>0$, depending only on $d,\ga,q,s,\ka_q$, such that if $\mu\in C_T^0\G_{\phi}^{\ka_q+\frac{q-1}{q}, \frac{2q}{q-1}}$ is a solution to \eqref{eq:cp}, for some $T>0$, satisfying
\begin{equation}\label{eq:qmonidr}
\|\mu^0\|_{\G_{\al}^{\ka_q,\frac{2q}{q-1}}} < \frac{\nu^2-2\beta}{C_q|\M|},
\end{equation}
then
\begin{equation}
\|\mu^t\|_{\G_{\phi^t}^{\ka_q,\frac{2q}{q-1}}} < \|\mu^{t'}\|_{\G_{\phi^{t'}}^{\ka_q,\frac{2q}{q-1}}} \qquad \forall 0\leq t' < t \leq T.
\end{equation}
Furthermore, if $\mu \in C_T^0\X_{\phi,\ga}^{\ka_r+\frac{2}{r},\ka_q+\frac{q-1}{q},r,\frac{2q}{q-1}}$, where $\ka_r>\ka_{0,r}$, such that
\begin{equation}
\|\mu^0\|_{\X_{\al,\ga}^{\ka_r,\ka_q,r,\frac{2q}{q-1}}} < \frac{\nu^2-2\beta}{|\M|\max\{C_r,C_q\}},
\end{equation}
then \eqref{eq:monr} also holds.
\end{prop}

\begin{remark}
Bounds for the thresholds $\ka_{0,r},\ka_{0,q}$ are explicitly worked out in the proof of \cref{prop:mon}. See the conditions \ref{U1},\ref{U2} and \ref{L1},\ref{L2} below together with their respectively ensuing analysis.
\end{remark}

The proof of \cref{prop:mon} consists of several lemmas. To begin, we observe from the chain rule and using equation \eqref{eq:cp} (there is an approximation step we omit),
\begin{multline}\label{eq:infmonrhs}
\frac{d}{dt} |e^{\phi^t(1+|\xi|^s)}\hat{\mu}^t(\xi)| = \Re\Bigg(|e^{\phi^t(1+|\xi|^s)}\hat{\mu}^t(\xi)|^{-1}\ol{e^{\phi^t (1+|\xi|^s)} \hat{\mu}^t(\xi)}\Bigg(\beta(1+|\xi|^s)e^{\phi^t (1+|\xi|^s)} \hat{\mu}^t(\xi)\\
-  e^{\phi^t(1+|\xi|^s)}\F(B^t(\mu^t,\mu^t))(\xi) - \frac{\nu^2}{2}(1+|\xi|^s)^2 e^{\phi^t(1+|\xi|^s)}\hat{\mu}^t \Bigg).
\end{multline}
Replacing the first and second terms by their magnitudes, we see that the right-hand side is $\leq$
\begin{equation}
 |e^{\phi^t(1+|\xi|^s)}\hat{\mu}^t(\xi)|\Bigg(\beta(1+|\xi|^s) - \frac{\nu^2}{2}(1+|\xi|^s)^2\Bigg) + \left|e^{\phi^t(1+|\xi|^s)}\F(B^t(\mu^t,\mu^t))(\xi)\right|.
\end{equation}
Using the elementary inequality (remember that $s\leq 1$)
\begin{equation}
\jp{\xi} \leq (1+|\xi|^s)^{\frac{1}{s}} \leq 2^{\frac{2-s}{2s}} \jp{\xi}
\end{equation}
together with our assumption that $\beta<\frac{\nu^2}{2}$, we arrive at the inequality
\begin{multline}\label{eq:infenid}
\frac{d}{dt}|e^{\phi^t(1+|\xi|^s)}\hat{\mu}^t(\xi)|  \leq -\jp{\xi}^{2s} |e^{\phi^t(1+|\xi|^s)}\hat{\mu}^t(\xi)|\paren*{\frac{\nu^2}{2}-\beta} + \left|e^{\phi^t(1+|\xi|^s)}\F(B^t(\mu^t,\mu^t))(\xi)\right|.
\end{multline}
It now follows from this identity, the chain rule, and differentiating inside the integral that for any $1\leq r<\infty$,
\begin{multline}\label{eq:enid}
\frac{1}{r}\frac{d}{dt} \| e^{\phi^t A^{1/2}}\mu^t\|_{\hat{W}^{\sigma s,r}}^r \leq -\paren*{\frac{\nu^2}{2} -\beta}\int_{\R^d}  \left|e^{\phi^t (1+|\xi|^s)}\jp{\xi}^{(\sigma+\frac{2}{r})s} \hat{\mu}^t(\xi) \right|^{r}d\xi\\
+ \int_{\R^d} \left|e^{\phi^t (1+|\xi|^s)}\jp{\xi}^{\sigma s} \hat{\mu}^t(\xi) \right|^{r-1}\jp{\xi}^{\sigma s} e^{\phi^t(1+|\xi|^s)}|\F(B^t(\mu^t,\mu^t))(\xi)|d\xi.
\end{multline}

We need to show that the second term in \eqref{eq:enid} is not so large that it cannot be absorbed by the first term, which is negative. This is a problem in Fourier analysis, which we address with the next two lemmas.

\begin{lemma}\label{lem:Bbndpre}
For any $t>0$ with $\phi^t - \nu W^t\geq 0$, it holds for any test functions $f,g$ that
\begin{multline}
|e^{\phi^t(1+|\xi|^s)}\F(B^t(f,g))(\xi)| \\
\lesssim_\ga |\M|\int_{\R^d}|\xi| |\eta|^{1-\ga} \left|e^{\phi^t(1+|\xi-\eta|^s)}\hat{f}(\xi-\eta) e^{\phi^t(1+|\eta|^s)}\hat{g}(\eta)\right| d\eta.
\end{multline}
\end{lemma}
\begin{proof}
We observe from the definition \eqref{eq:Bdef} of $B^t$ that for any test functions $f,g$,
\begin{multline}
e^{\phi^t(1+|\xi|^s)}\F(B^t(f,g))(\xi) \\
= \int_{\R^d}e^{\phi^t(1+|\xi|^s) - \nu W^t(|\xi|^s - |\eta|^s - |\xi-\eta|^s-1)} \paren*{\xi\cdot\M\eta}\hat{\g}(\eta) \hat{f}(\xi-\eta)\hat{g}(\eta) d\eta.
\end{multline}
Writing $1=e^{\phi^t(1+|\eta|^s)}e^{-\phi^t(1+|\eta|^s)} = e^{\phi^t(1+|\xi-\eta|^s)}e^{-\phi^t(1+|\xi-\eta|^s)}$, we see that the magnitude of the preceding right-hand side is controlled by
\begin{multline}
\left|\int_{\R^d} e^{(\phi^t-\nu W^t)(|\xi|^s-|\eta|^s-|\xi-\eta|^s-1)}\paren*{\xi\cdot\M\eta}\hat{\g}(\eta) e^{\phi^t(1+|\xi-\eta|^s)}\hat{f}(\xi-\eta) e^{\phi^t(1+|\eta|^s)}\hat{g}(\eta)d\eta \right| \\
\lesssim_\ga |\M|\int_{\R^d} e^{(\phi^t-\nu W^t)(|\xi|^s-|\eta|^s-|\xi-\eta|^s-1)} |\xi| |\eta|^{1-\ga}\left|e^{\phi^t(1+|\xi-\eta|^s)}\hat{f}(\xi-\eta) e^{\phi^t(1+|\eta|^s)}\hat{g}(\eta)\right|d\eta,
\end{multline}
where we have used our assumption that $|\eta\hat{\g}(\eta)|\lesssim_{\ga} |\eta|^{1-\ga}$. Since $\phi^t-\nu W^t\geq 0$ by assumption and $|\xi|^s-|\eta|^s-|\xi-\eta|^s-1\leq 0$ by $\|\cdot\|_{\ell^1} \leq \|\cdot\|_{\ell^s}$ (recall $s\leq 1$), the desired inequality now follows. 
\end{proof}

Next, we observe that by applying \cref{lem:Bbndpre} to the second term in the right-hand side of \eqref{eq:enid}, we need to estimate expressions of the form
\begin{equation}
\int_{\R^d} \left|e^{\phi^t (1+|\xi|^s)}\jp{\xi}^{\ka s} \hat{h}(\xi) \right|^{r-1}\jp{\xi}^{\ka s+1}\int_{\R^d}|\eta|^{1-\ga} \left|e^{\phi^t(1+|\xi-\eta|^s)}\hat{f}(\xi-\eta) e^{\phi^t(1+|\eta|^s)}\hat{g}(\eta)\right| d\eta d\xi,
\end{equation}
where $f,g,h$ are test functions. We take care of such expressions with the next lemma.

\begin{lemma}\label{lem:Bbnd}
Let $d\geq 1$, $0<\ga<d+1$, $1\leq r\leq \infty$, $\frac{1}{2}< s\leq 1$. If $\ga>1$, also assume that $1\leq q<\frac{d}{\ga-1}$. Then there exists a threshold $\ka_0$ depending on $d,\ga,r,s$, such that for any $\ka>\ka_0$, there exists a constant $C>0$ depending on $d,\ga,r,s,q,\ka$ so that
\begin{multline}\label{eq:Bbnd}
\int_{\R^d} \left|e^{\phi^t (1+|\xi|^s)}\jp{\xi}^{\ka s} \hat{h}(\xi) \right|^{r-1}\jp{\xi}^{\ka s+1}\int_{\R^d}|\eta|^{1-\ga} \left|e^{\phi^t(1+|\xi-\eta|^s)}\hat{f}(\xi-\eta) e^{\phi^t(1+|\eta|^s)}\hat{g}(\eta)\right| d\eta d\xi \\
\leq C\|e^{\phi^t A^{1/2}} h\|_{\hat{W}^{(\ka+\frac{2}{r})s,r}}^{r-1}\Bigg(\|e^{\phi^t A^{1/2}}f\|_{\hat{W}^{(\ka-\frac{2(r-1)}{r})s+1, r}} \|e^{\phi^t A^{1/2}} g\|_{\hat{W}^{1-\ga,\frac{2q}{q-1}}}\indic_{\ga>1} \\
+ \|e^{\phi^t A^{1/2}}f\|_{\hat{W}^{(\ka + \frac{2}{r})s,r}} \|e^{\phi^t A^{1/2}}g\|_{\hat{W}^{\ka s,r}} + \|e^{\phi^t A^{1/2}}f\|_{\hat{W}^{\ka s,r}}\|e^{\phi^t A^{1/2}}g\|_{\hat{W}^{(\ka + \frac{2}{r})s,r}}\Bigg).
\end{multline}
\end{lemma}
\begin{proof}
Writing $\jp{\xi}^{\ka s+1} = \jp{\xi}^{\ka s+1-\frac{2s(r-1)}{r}}\jp{\xi}^{\frac{2s(r-1)}{r}}$ and applying H\"older's inequality with conjugate exponents $\frac{r}{r-1}$ and $r$, we find that the left-hand side of \eqref{eq:Bbnd} is $\leq$
\begin{multline}\label{eq:Bbndstart}
\|e^{\phi^t A^{1/2}}\jp{\nabla}^{(\ka + \frac{2}{r})s} h\|_{\hat{L}^r}^{r-1}\Bigg(\int_{\R^d} \jp{\xi}^{(r\ka - 2(r-1))s+r} \left|\int_{\R^d}|\eta|^{1-\ga} \right.\\
\left.\left|e^{\phi^t(1+|\xi-\eta|^s)}\hat{f}(\xi-\eta) e^{\phi^t(1+|\eta|^s)}\hat{g}(\eta)\right| d\eta \right|^r d\xi\Bigg)^{1/r},
\end{multline}
with obvious modification if $r=\infty$. We need to estimate the second factor. We first address the singularity in $\eta$ at low frequency if $1-\ga<0$. Observe that by separately considering the regions $|\xi| \leq 2|\eta|$ and $|\xi|>2|\eta|$, it follows from Young's inequality that
\begin{align}
&\Bigg(\int_{\R^d} \jp{\xi}^{(r\ka - 2(r-1))s+r} \left|\int_{|\eta|\leq 1}|\eta|^{1-\ga}\left|\paren*{e^{\phi^t(1+|\xi-\eta|^s)}\hat{f}(\xi-\eta)} \paren*{e^{\phi^t(1+|\eta|^s)}\hat{g}(\eta)}\right| d\eta \right|^r d\xi\Bigg)^{1/r} \nn\\
&\lesssim_{r,\ka,s,\ga,d} \|\jp{\cdot}^{(\ka -\frac{2(r-1)}{r})s+1} e^{\phi^t(1+|\cdot|^s)}\hat{f}\|_{L^r} \||\cdot|^{1-\ga}e^{\phi^t(1+|\cdot|^s)} \hat{g}1_{B(0,1)}\|_{L^1} \nn\\
&\lesssim \|e^{\phi^t A^{1/2}}f\|_{\hat{W}^{(\ka-\frac{2(r-1)}{r})s+1, r}} \|e^{\phi^t A^{1/2}} g\|_{\hat{W}^{1-\ga,\frac{2q}{q-1} }}, \label{eq:monloga}
\end{align}
for any $1\leq q<\frac{d}{\ga-1}$, where we use H\"older's inequality on the second factor to obtain the ultimate line.

Next, we make an elementary Bony decomposition by splitting the space of $(\xi,\eta)$ into the regions $|\xi-\eta| \leq \frac{|\eta|}{8}$, $|\eta|\leq \frac{|\xi-\eta|}{8}$, and $\frac{1}{8}< \frac{|\xi-\eta|}{|\eta|}<8$.

\medskip
\textbullet \ If $|\xi-\eta| \leq \frac{|\eta|}{8}$, then $|\eta| \sim |\xi|$. So by Young's inequality followed by application of \cref{lem:Sob}, it holds for any $1\leq p\leq r$ that
\begin{align}
&\paren*{\int_{\R^d}\jp{\xi}^{r( (\ka-\frac{2(r-1)}{r}) s+1)}\paren*{\int_{\substack{|\xi-\eta|\leq \frac{|\eta|}{8} \\ |\eta|>1}}e^{\phi^t(1+|\xi-\eta|^s)} |\hat{f}(\xi-\eta)| |\eta|^{1-\gamma}e^{\phi^t(1+|\eta|^s)}|\hat{g}(\eta)| d\eta}^r d\xi }^{1/r} \nn\\
&\lesssim_{r,\ka,s,\ga} \|e^{\phi^t A^{1/2}} f\|_{\hat{L}^{\frac{pr}{p(r+1)-r}}} \|\jp{\nabla}^{(\ka-\frac{2(r-1)}{r})s+2-\ga}e^{\phi^t A^{1/2}}g\|_{\hat{L}^p} \nn\\
&\lesssim_{d,r,s,\ka,\ga,p} \|e^{\phi^t A^{1/2}} f\|_{\hat{W}^{0,1}} \|e^{\phi^t A^{1/2}} g\|_{\hat{W}^{\ka s+2-\ga,1}} \indic_{r=1} \nn\\
&\ph+ \|e^{\phi^t A^{1/2}} f\|_{\hat{W}^{0,r}} \|e^{\phi^t A^{1/2}} g\|_{\hat{W}^{((\ka-\frac{2(r-1)}{r})s+2+\frac{d(r-1)}{r}-\ga)+, r}}\indic_{\substack{r>1 \\ p=1}} \nn\\
&\ph+ \|e^{\phi^t A^{1/2}} f\|_{\hat{W}^{\frac{d(r-1)}{r}+, r}} \|e^{\phi^t A^{1/2}} g\|_{\hat{W}^{(\ka- \frac{2(r-1)}{r})s+2-\ga, r}}\indic_{\substack{r>1 \\ p=r}} \nn\\
&\ph + \|e^{\phi^t A^{1/2}} f\|_{\hat{W}^{\frac{d(p-1)}{p}+,r}}\|e^{\phi^t A^{1/2}} g\|_{\hat{W}^{((\ka-\frac{2(r-1)}{r})s+2-\ga+\frac{d(r-p)}{rp})+, r}} \indic_{\substack{r>1 \\ 1<p<r}}. \label{eq:monlh}
\end{align}

\medskip
\textbullet \ If $|\eta|\leq \frac{|\xi-\eta|}{8}$, then $|\xi-\eta|\sim |\xi|$. Again using Young's inequality and \cref{lem:Sob}, it holds for any $1\leq \tl{p}\leq r$ that
\begin{align}
&\paren*{\int_{\R^d}\jp{\xi}^{r( (\ka-\frac{2(r-1)}{r}) s+1)}\paren*{\int_{\substack{|\eta|\leq \frac{|\xi-\eta|}{8} \\ |\eta|>1}}e^{\phi^t(1+|\xi-\eta|^s)} |\hat{f}(\xi-\eta)| |\eta|^{1-\gamma}e^{\phi^t(1+|\eta|^s)}|\hat{g}(\eta)| d\eta}^r d\xi }^{1/r} \nn\\
&\lesssim_{r,\ka,s,\ga} \|\jp{\nabla}^{(\ka-\frac{2(r-1)}{r})s +1} e^{\phi^t A^{1/2}}f\|_{\hat{L}^{\frac{\tl{p}r}{\tl{p}(r+1)-r}}} \|\jp{\nabla}^{1-\ga}e^{\phi^t A^{1/2}} g\|_{\hat{L}^{\tl{p}}} \nn\\
&\lesssim_{d,s,\ka,\ga,r,\tl{p}} \|e^{\phi^t A^{1/2}} f\|_{\hat{W}^{\ka s +1,1}}\|e^{\phi^t A^{1/2}} g\|_{\hat{W}^{1-\ga,1}} \indic_{r=1} \nn\\
&\ph + \|e^{\phi^t A^{1/2}} f\|_{\hat{W}^{(\ka-\frac{2(r-1)}{r})s + 1,r}}\|e^{\phi^t A^{1/2}} g\|_{\hat{W}^{(1-\ga+\frac{d(r-1)}{r})+,r}}\indic_{\substack{r>1 \\ \tl{p}=1}} \nn\\
&\ph + \|e^{\phi^t A^{1/2}} f\|_{\hat{W}^{((\ka-\frac{2(r-1)}{r})s + 1+\frac{d(r-1)}{r})+,r}}\|e^{\phi^t A^{1/2}} g\|_{\hat{W}^{1-\ga,r}}\indic_{\substack{r>1 \\ \tl{p}=r}} \nn\\
&\ph +  \|e^{\phi^t A^{1/2}} f\|_{\hat{W}^{((\ka-\frac{2(r-1)}{r})s + 1+\frac{d(\tl{p}-1)}{\tl{p}})+,r}}\|e^{\phi^t A^{1/2}} g\|_{\hat{W}^{(1-\ga + \frac{d(r-\tl{p})}{r\tl{p}})+,r}}\indic_{\substack{r>1 \\ 1<\tl{p}<r}}. \label{eq:monhl}
\end{align}

\medskip
\textbullet \ If $\frac{1}{8}< \frac{|\xi-\eta|}{|\eta|}<8$, then $\min\{|\xi-\eta|,|\eta|\} \gtrsim |\xi|$. So we can evenly distribute the derivatives between $f$ and $g$ and use Young's inequality together with \cref{lem:Sob} to obtain
\begin{align}
&\paren*{\int_{\R^d}\jp{\xi}^{r( (\ka-\frac{2(r-1)}{r}) s+1)}\paren*{\int_{\substack{ \frac{1}{8} \leq \frac{|\xi-\eta|}{|\eta|} \leq 8 \\ |\eta|>1}}e^{\phi^t(1+|\xi-\eta|^s)} |\hat{f}(\xi-\eta)| |\eta|^{1-\gamma}e^{\phi^t(1+|\eta|^s)}|\hat{g}(\eta)| d\eta}^r d\xi }^{1/r} \nn\\
&\lesssim_{r,\ka,s,\ga} \|\jp{\nabla}^{\frac{(\ka-\frac{2(r-1)}{r})s +2-\ga}{2}} e^{\phi^t A^{1/2}}f\|_{\hat{L}^{\frac{2r}{r+1}}} \|\jp{\nabla}^{\frac{(\ka-\frac{2(r-1)}{r})s +2-\ga}{2}} e^{\phi^t A^{1/2}}g\|_{\hat{L}^{\frac{2r}{r+1}}} \nn\\
&\lesssim_{d,r,\ka,s,\ga} \|e^{\phi^t A^{1/2}} f\|_{\hat{W}^{\frac{\ka s +2-\ga}{2},1}} \|e^{\phi^t A^{1/2}} g\|_{\hat{W}^{\frac{\ka s +2-\ga}{2},1}}  \indic_{r=1} \nn\\
&\ph +\|e^{\phi^t A^{1/2}} f\|_{\hat{W}^{(\frac{(\ka-\frac{2(r-1)}{r})s +2 + \frac{d(r-1)}{r}-\ga}{2})+,r}}\|e^{\phi^t A^{1/2}} g\|_{\hat{W}^{(\frac{(\ka-\frac{2(r-1)}{r})s +2 + \frac{d(r-1)}{r}-\ga}{2})+,r}} \indic_{r>1}. \label{eq:monhh}
\end{align}

\medskip

Combining the estimates \eqref{eq:monloga}, \eqref{eq:monlh}, \eqref{eq:monhl}, \eqref{eq:monhh}, we see that
\begin{multline}\label{eq:Bnd}
\Bigg(\int_{\R^d} \jp{\xi}^{(r\ka - 2(r-1))s+r} \left|\int_{\R^d}|\eta|^{1-\ga} \left|\paren*{e^{\phi^t(1+|\xi-\eta|^s)}\hat{f}(\xi-\eta)} \paren*{e^{\phi^t(1+|\eta|^s)}\hat{g}(\eta)}\right| d\eta \right|^r d\xi\Bigg)^{1/r} \\
\lesssim_{d,s,\ga,\ka,r,q,p,\tl{p}} \|e^{\phi^t A^{1/2}}f\|_{\hat{W}^{(\ka-\frac{2(r-1)}{r})s+1, r}} \|e^{\phi^t A^{1/2}} g\|_{\hat{W}^{1-\ga,\frac{2q}{q-1}}}\indic_{\ga>1} \\
+ \Bigg(\|e^{\phi^t A^{1/2}} f\|_{\hat{W}^{0,1}} \|e^{\phi^t A^{1/2}} g\|_{\hat{W}^{\ka s+2-\ga,1}} +  \|e^{\phi^t A^{1/2}} f\|_{\hat{W}^{\ka s +1,1}}\|e^{\phi^t A^{1/2}} g\|_{\hat{W}^{1-\ga,1}} \\
+ \|e^{\phi^t A^{1/2}} f\|_{\hat{W}^{\frac{\ka s +2-\ga}{2},1}} \|e^{\phi^t A^{1/2}} g\|_{\hat{W}^{\frac{\ka s +2-\ga}{2},1}}  \Bigg)\indic_{r=1} \\
+\|e^{\phi^t A^{1/2}} f\|_{\hat{W}^{\frac{d(p-1)}{p}+,r}}\|e^{\phi^t A^{1/2}} g\|_{\hat{W}^{((\ka-\frac{2(r-1)}{r})s+2-\ga+\frac{d(r-p)}{rp})+, r}} \indic_{\substack{r>1 \\ 1<p<r}}\\
+ \|e^{\phi^t A^{1/2}} f\|_{\hat{W}^{((\ka-\frac{2(r-1)}{r})s + 1+\frac{d(\tl{p}-1)}{\tl{p}})+,r}}\|e^{\phi^t A^{1/2}} g\|_{\hat{W}^{(1-\ga + \frac{d(r-\tl{p})}{r\tl{p}})+,r}}\indic_{\substack{r>1 \\ 1<\tl{p}<r}} \\
+\|e^{\phi^t A^{1/2}} f\|_{\hat{W}^{(\frac{(\ka-\frac{2(r-1)}{r})s +2 + \frac{d(r-1)}{r}-\ga}{2})+,r}}\|e^{\phi^t A^{1/2}} g\|_{\hat{W}^{(\frac{(\ka-\frac{2(r-1)}{r})s +2 + \frac{d(r-1)}{r}-\ga}{2})+,r}} \indic_{r>1}.
\end{multline}
Above, we have limited ourselves to the cases $r=1$ and $r>1,\ 1<p,\tl{p} <\infty$ so as to simplify the exposition (the cost is an insignificant $\vep$ loss at the endpoint exponents). In order to obtain the desired estimate \eqref{eq:Bbnd}, we need the maximum Sobolev index of the norms in \eqref{eq:Bnd} to be $<(\ka+\frac{2}{r})s$. This leads us to make the following assumptions on the parameters $d,s,\ka,\ga,r,p,\tl{p}$.
\begin{enumerate}[(U1)]
\item\label{U1} $r=1$
\begin{enumerate}
\item\label{1a}
$\ka s +2-\ga \leq (\ka+2)s$,
\item\label{1b}
$\max\{\ka s+1,1-\ga\} \leq (\ka+2)s$,
\item\label{1c}
$\frac{\ka s+2-\ga}{2} \leq (\ka+2)s$.
\end{enumerate}
\item\label{U2} $r>1$
\begin{enumerate}
\item\label{2a}
$(\ka - \frac{2(r-1)}{r})s + 1 \leq (\ka + \frac{2}{r})s$;
\item\label{2b}
there exists $p\in (1,r)$ such that $\max\{\frac{d(p-1)}{p}, (\ka - \frac{2(r-1)}{r})s + 2-\ga + \frac{d(r-p)}{rp}\} < (\ka + \frac{2}{r})s$;
\item\label{2c}
there exists $\tl{p}\in (1,r)$ such that $\max\{(\ka-\frac{2(r-1)}{r})s + 1+\frac{d(\tl{p}-1)}{\tl{p}},1-\ga + \frac{d(r-\tl{p})}{r\tl{p}} \} < (\ka+\frac{2}{r})s$;
\item\label{2d}
$\frac{(\ka-\frac{2(r-1)}{r})s +2 + \frac{d(r-1)}{r}-\ga}{2} < (\ka+\frac{2}{r})s$.
\end{enumerate}
\end{enumerate}
Let us analyze the above conditions.

For \ref{1a}, observe
\begin{equation}
\ka s + 2-\ga \leq (\ka+2)s \Longleftrightarrow \frac{2-\ga}{2}\leq s.
\end{equation}
Since $\ga>0$ by assumption, we can always choose $s$ sufficiently close to $1$, so that this condition holds. For \ref{1b}, the inequality is equivalent to
\begin{equation}
\frac{1}{2}\leq s \quad \text{and} \quad 1-\ga-2s \leq \ka s.
\end{equation}
The first inequality is true by assumption, and the second inequality holds by taking $\ka$ sufficiently large depending on given $\ga,s$. \ref{1c} is equivalent to
\begin{equation}
2-\ga-4s \leq \ka s,
\end{equation}
which holds by taking $\ka$ sufficiently large depending on given $s,\ga$. \ref{2a} is equivalent to
\begin{equation}
-\frac{2(r-1)s}{r} + 1 \leq \frac{2s}{r} \quad\Longleftrightarrow\quad \frac{1}{2} \leq s,
\end{equation}
which holds by assumption. For \ref{2b}, observe that
\begin{equation}
\frac{d(p-1)}{p} < (\ka + \frac{2}{r})s \quad \Longleftrightarrow\quad d-(\ka + \frac{2}{r})s < \frac{d}{p}
\end{equation}
and
\begin{equation}
(\ka - \frac{2(r-1)}{r})s + 2-\ga + \frac{d(r-p)}{rp} < (\ka + \frac{2}{r})s \quad\Longleftrightarrow\quad \frac{d}{p} < \frac{d}{r} +\ga+2s-2.
\end{equation}
Thus, it is possible to find such a $p\in (1,r)$ if and only if
\begin{align}
d-(\ka + \frac{2}{r})s < \frac{d}{r} +\ga+2s-2\quad \Longleftrightarrow\quad  d-\frac{2s}{r}-\frac{d}{r}-\ga-2s+2 < \ka s,
\end{align}
which holds by taking $\ka$ sufficiently large depending on given $d,\ga,r,s$. For \ref{2c}, observe
\begin{equation}
(\ka-\frac{2(r-1)}{r})s + 1+\frac{d(\tl{p}-1)}{\tl{p}} <(\ka+\frac{2}{r})s \quad \Longleftrightarrow \quad d+1-2s < \frac{d}{\tl{p}}
\end{equation}
and
\begin{equation}
1-\ga+\frac{d(r-\tl{p})}{r\tl{p}} < (\ka+\frac{2}{r})s \quad \Longleftrightarrow \quad \frac{d}{\tl{p}} < \ka s + \frac{2s+d}{r}+\ga-1.
\end{equation}
It is possible to find such a $\tl{p}\in (1,r)$ if and only if
\begin{equation}
d+1-2s < \ka s + \frac{2s+d}{r}+\ga-1 \quad \Longleftrightarrow \quad d+2-2s-\frac{2s+d}{r} - \ga < \ka s,
\end{equation}
which holds by taking $\ka$ sufficiently large depending on given $d,\ga,r,s$. Lastly, \ref{2d} is equivalent to
\begin{equation}
d+2 -\ga-2s-\frac{d+2s}{r}<\ka s,
\end{equation}
which holds by taking $\ka$ sufficiently large depending on given $d,\ga,r,s$.

Next, we observe that in order to obtain the desired estimate \eqref{eq:Bbnd}, we need the minimum Sobolev index of the norms appearing in \eqref{eq:Bnd} to be $\leq \ka s$. This leads us to make the following additional assumptions on the parameters $d,s,\ka,\ga,r,p,\tl{p}$:
\begin{enumerate}[(L1)]
\item\label{L1} $r=1$
	\begin{enumerate}
	\item\label{1'a} $1-\ga \leq \ka s$ 
	\item\label{1'b} $\frac{\ka s+2-\ga}{2} \leq \ka s$ 
	\end{enumerate}
\item\label{L2} $r>1$
	\begin{enumerate}
	\item\label{2'a} there exists $p\in (1,r)$ such that $\min\{\frac{d(p-1)}{p}, (\ka - \frac{2(r-1)}{r})s + 2-\ga + \frac{d(r-p)}{rp}\} < \ka s$ 
	\item\label{2'b} there exists $\tl{p}\in (1,r)$ such that $\min\{(\ka - \frac{2(r-1)}{r})s + 1+\frac{d(\tl{p}-1)}{\tl{p}}, 1-\ga + \frac{d(r-\tl{p})}{r\tl{p}}\} < \ka s$
	\item\label{2'c} $\frac{(\ka - \frac{2(r-1)}{r})s+2+\frac{d(r-1)}{r}-\ga}{2} < \ka s$.
	\end{enumerate}
\end{enumerate}
Let us analyze the preceding assumptions.

For any $\frac{1}{2}<s\leq 1$, \ref{1'a} always holds if we assume $2(1-\ga) \leq \ka$. For \ref{1'b},
\begin{equation}
\frac{\ka s+2-\ga}{2} \leq \ka s \quad \Longleftrightarrow \quad 2-\ga \leq \ka s,
\end{equation}
which is ensured by taking $\ka$ sufficiently large depending on given $\ga,s$. For \ref{2'a}, we need
\begin{equation}
\frac{d(p-1)}{p} < \ka s \quad \text{or} \quad (\ka - \frac{2(r-1)}{r})s + 2-\ga + \frac{d(r-p)}{rp} < \ka s.
\end{equation}
Since $p<r$ and therefore $\frac{d(p-1)}{p} < \frac{d(r-1)}{r}$, the first inequality is valid if $\frac{d(r-1)}{r}<\ka s$, which holds by taking $\ka$ sufficiently large depending on given $d,r,s$. For the second inequality, we see
\begin{align}
(\ka - \frac{2(r-1)}{r})s + 2-\ga + \frac{d(r-p)}{rp} < \ka s \quad \Longleftrightarrow \quad 2-2s-\ga + \frac{d}{p}+\frac{(2s-d)}{r} < 0.
\end{align}
The left-hand side of the second inequality is maximized by choosing $p=1$, so it would suffice to assume $d,\ga,s,r$ satisfy
\begin{equation}
2-2s-\ga+d+\frac{(2s-d)}{r}<0.
\end{equation}
For \ref{2'b}, we need
\begin{equation}
(\ka - \frac{2(r-1)}{r})s + 1+\frac{d(\tl{p}-1)}{\tl{p}} < \ka s \quad \text{or} \quad 1-\ga + \frac{d(r-\tl{p})}{r\tl{p}} < \ka.
\end{equation}
The first inequality is equivalent to
\begin{align}
d+1-2s +\frac{2s}{r} - \frac{d}{\tl{p}} < 0.
\end{align}
The left-hand side is maximized by choosing $\tl{p}=r$, therefore it suffices to assume
\begin{equation}
d+1-2s+\frac{(2s-d)}{r}<0.
\end{equation}
The second inequality is equivalent to
\begin{equation}
1-\ga+\frac{d}{\tl{p}}-\frac{d}{r}<\ka s,
\end{equation}
the left-hand side of which is maximized by choosing $\tl{p}=1$. So, it would suffice to assume
\begin{equation}
1-\ga + d-\frac{d}{r}<\ka s,
\end{equation}
which is seen to hold by choosing $\ka$ sufficiently large depending on given $d,\ga,r,s$. For \ref{2'c}, we observe that the inequality is equivalent to
\begin{equation}
d+2-\ga-2s+\frac{2s-d}{r} < \ka s,
\end{equation}
which is valid provided that $\ka$ is sufficiently large depending on given $d,\ga,r,s$.

The preceding sets of assumptions tell us that given $d,\ga,r,s$, there is a threshold $\ka_0$ depending on $d,\ga,r,s$, such that for all $\ka>\ka_0$, the right-hand side of \eqref{eq:Bnd} is controlled by
\begin{multline}
\|e^{\phi^t A^{1/2}}f\|_{\hat{W}^{(\ka-\frac{2(r-1)}{r})s+1, r}} \|e^{\phi^t A^{1/2}} g\|_{\hat{W}^{1-\ga,\frac{2q}{q-1}}}\indic_{\ga>1} \\
+ \|e^{\phi^t A^{1/2}}f\|_{\hat{W}^{(\ka + \frac{2}{r})s,r}} \|e^{\phi^t A^{1/2}}g\|_{\hat{W}^{\ka s,r}} + \|e^{\phi^t A^{1/2}}f\|_{\hat{W}^{\ka s,r}}\|e^{\phi^t A^{1/2}}g\|_{\hat{W}^{(\ka + \frac{2}{r})s,r}}.
\end{multline}
Recalling the starting inequality \eqref{eq:Bbndstart}, we see that the proof is complete.
\end{proof}

We are now prepared to prove \cref{prop:mon}.
\begin{proof}[Proof of \cref{prop:mon}]
Given $1< q < \frac{d}{\ga-1}$, set $q'\coloneqq \frac{q}{q-1}$. We first show that the quantity
\begin{equation}
\|e^{\phi^t A^{1/2}}\mu^t\|_{\hat{W}^{\ka s, 2q'}}^{2q'}
\end{equation}
is strictly decreasing on an interval $[0,T]$, provided that $\ka$ sufficiently large depending on $q'$ and the higher norm $\|e^{\phi^t A^{1/2}}\mu^t\|_{\hat{W}^{(\ka + \frac{1}{q'})s, 2q'}}$ remains finite on $[0,T]$. Indeed, let $\ka_{0,2q'}$ denote the threshold given by \cref{lem:Bbnd} with $r=2q'$, and suppose that $\ka>\ka_{0,2q'}$. Using H\"older's inequality and \Cref{lem:Bbndpre,lem:Bbnd}, we see that
\begin{multline}
\int_{\R^d} \left|e^{\phi^t (1+|\xi|^s)}\jp{\xi}^{\ka s} \hat{\mu}^t(\xi) \right|^{2q'-1}\jp{\xi}^{\ka s} e^{\phi^t(1+|\xi|^s)}|\F(B^t(\mu^t,\mu^t))(\xi)|d\xi \\
\leq C|\M|\|e^{\phi^t A^{1/2}} \mu^t\|_{\hat{W}^{(\ka+\frac{1}{q'})s,2q'}}^{2q'-1}\Bigg(\|e^{\phi^t A^{1/2}}\mu^t\|_{\hat{W}^{(\ka-2+\frac{1}{q'})s+1, 2q'}} \|e^{\phi^t A^{1/2}} \mu^t\|_{\hat{W}^{1-\ga,2q'}}\indic_{\ga>1} \\
+ \|e^{\phi^t A^{1/2}}\mu^t\|_{\hat{W}^{(\ka + \frac{1}{q'})s,2q'}} \|e^{\phi^t A^{1/2}}\mu^t\|_{\hat{W}^{\ka s,2q'}}\Bigg)
\end{multline}
where the constant $C>0$ depends only on $d,\ga,q,s,\ka$. Applying this bound to the differential identity \eqref{eq:enid} (with $r$ replaced by $2q'$), it follows that
\begin{multline}
\frac{d}{dt} \frac{1}{2q'} \|e^{\phi^t A^{1/2}} \mu^t\|_{\hat{W}^{\ka s,2q'}}^{2q'} \leq -\paren*{\frac{\nu^2}{2} -\beta}\|e^{\phi^t A^{1/2}}\mu^t\|_{\hat{W}^{(\ka+\frac{1}{q'})s,2q'}}^{2q'} \\
+C|\M|\|e^{\phi^t A^{1/2}} \mu^t\|_{\hat{W}^{(\ka+\frac{1}{q'})s,2q'}}^{2q'-1}\Bigg(\|e^{\phi^t A^{1/2}}\mu^t\|_{\hat{W}^{(\ka-2+\frac{1}{q'})s+1, 2q'}} \|e^{\phi^t A^{1/2}} g\|_{\hat{W}^{1-\ga,2q'}}\indic_{\ga>1} \\
+ \|e^{\phi^t A^{1/2}}\mu^t\|_{\hat{W}^{(\ka + \frac{1}{q'})s,2q'}} \|e^{\phi^t A^{1/2}}\mu^t\|_{\hat{W}^{\ka s,2q'}} \Bigg).
\end{multline}
Note that since $s>\frac{1}{2}$ by assumption, $(\ka-2+\frac{1}{q'})s + 1 < (\ka + \frac{1}{q'})s$. Thus, the right-hand side of the preceding inequality is $\leq$
\begin{equation}\label{eq:wtsnn}
\|e^{\phi^t A^{1/2}}\mu^t\|_{\hat{W}^{(\ka+\frac{1}{q'})s,2q'}}^{2q'}\Bigg(C|\M|\|e^{\phi^t A^{1/2}}\mu^t\|_{\hat{W}^{\ka s,2q'}} -\paren*{\frac{\nu^2}{2} -\beta}\Bigg).
\end{equation}
We now want to show that if $\|e^{\al A^{1/2}}\mu^0\|_{\hat{W}^{\ka s, 2q'}} < \frac{\nu^2 - 2\beta}{2C|\M|}$ and the first factor remains finite on $[0,T]$, then $\|e^{\phi^t A^{1/2}}\mu^0\|_{\hat{W}^{\ka s, 2q'}}$ is strictly decreasing on $[0,T]$. To do this, we use a continuity argument.

Suppose that monotonicity does not hold over $[0,T]$. Then since the function $t\mapsto \|e^{\phi^t A^{1/2}}\mu^t\|_{\hat{W}^{\ka s,2q'}}$ is continuous on $[0,T]$, there exists a minimal time $T_*\in (0,T]$ such that
\begin{equation}\label{eq:Tcontra}
\|e^{\phi^{T_*} A^{1/2}}\mu^{T_*}\|_{\hat{W}^{\ka s, 2q'}} \geq \frac{\nu^2 - 2\beta}{2C|\M|}.
\end{equation}
By continuity and the extreme value theorem, we in fact have an equality in the preceding relation. By minimality of $T_*$,
\begin{equation}\label{eq:ubhold}
\|e^{\phi^t A^{1/2}}\mu^t\|_{\hat{W}^{\ka s, 2q'}} < \frac{\nu^2-2\beta}{2C|\M|} \qquad \forall t\in [0,T_*),
\end{equation}
which, by \eqref{eq:wtsnn} and the fundamental theorem of calculus, implies that there exists an $\vep>0$ such that
\begin{equation}
\|e^{\al A^{1/2}}\mu^0\|_{\hat{W}^{\ka s, 2q'}}^{2q'} - \|e^{\phi^t A^{1/2}}\mu^t\|_{\hat{W}^{\ka s, 2q'}}^{2q'} \geq \vep \qquad \forall t\in [0,T_*].
\end{equation}
But this evidently contradicts \eqref{eq:Tcontra}. Thus, the inequality \eqref{eq:ubhold} holds on $[0,T]$, which implies that the right-hand side of inequality \eqref{eq:wtsnn} is negative on $[0,T]$, as desired.

\medskip
We next show that this monotonicity property of the norm $\|e^{\phi^t A^{1/2}}\mu^t\|_{\hat{W}^{\ka s, 2q'}}$ also implies a monotonicity property for the norm $\|e^{\phi^t A^{1/2}}\mu^t\|_{\hat{W}^{\tl{\ka} s,r}}$, for appropriate $\tl{\ka}$, provided that $\|e^{\al A^{1/2}}\mu^0\|_{\hat{W}^{\tl{\ka} s, r}}$ is sufficiently small. If $\ga\leq 1$, then this step is unnecessary and the argument given above suffices with $2q'$ replaced by $r$.

Let $\ka_{0,r}$ denote the regularity threshold given by \cref{lem:Bbnd}, and let $\ka_r > \ka_{0,r}$. Again using H\"older's inequality and \Cref{lem:Bbndpre,lem:Bbnd}, we see that
\begin{multline}
\int_{\R^d} \left|e^{\phi^t (1+|\xi|^s)}\jp{\xi}^{\ka_r s} \hat{\mu}^t(\xi) \right|^{r-1}\jp{\xi}^{\ka_r s} e^{\phi^t(1+|\xi|^s)}|\F(B^t(\mu^t,\mu^t))(\xi)|d\xi \\
\leq C_{r,q}|\M|\|e^{\phi^t A^{1/2}} \mu^t\|_{\hat{W}^{(\ka_r+\frac{2}{r})s,r}}^{r-1}\Bigg(\|e^{\phi^t A^{1/2}}\mu^t\|_{\hat{W}^{(\ka_r-2+\frac{2}{r})s+1, r}} \|e^{\phi^t A^{1/2}} \mu^t\|_{\hat{W}^{1-\ga,2q'}}\indic_{\ga>1} \\
+ \|e^{\phi^t A^{1/2}}\mu^t\|_{\hat{W}^{(\ka_r + \frac{2}{r})s,r}} \|e^{\phi^t A^{1/2}}\mu^t\|_{\hat{W}^{\ka_r s,r}} \Bigg),
\end{multline}
where the constant $C_{r,q}>0$ depends only on $d,\ga,r,q,s,\ka_r$. We use the subscript $r$ to emphasize the dependence on $r$, as we shall momentarily invoke another constant and regularity parameter depending on $q$. Applying the preceding bound to the differential identity \eqref{eq:enid} it follows that
\begin{multline}
\frac{d}{dt}\frac{1}{r}\|e^{\phi^t A^{1/2}}\mu^t\|_{\hat{W}^{\ka_r s,r}}^r \leq \Bigg(C_{r,q}|\M|\Big(\|e^{\phi^t A^{1/2}}\mu^t\|_{\hat{W}^{\ka_r s,r}} + \|e^{\phi^t A^{1/2}} \mu^t\|_{\hat{W}^{1-\ga,2q'}}\indic_{\ga>1}\Big)\\
-\paren*{\frac{\nu^2}{2} -\beta}\Bigg) \|e^{\phi^t A^{1/2}}\mu^t\|_{\hat{W}^{(\ka_r+\frac{2}{r})s,r}}^r.
\end{multline}
Since $1-\ga < 0$ if $\ga>1$, we know that there is a constant $C_q$ depending on $d,\ga,q,s,\ka_q$, for $\ka_q >\ka_{0,2q'}$, such that if
\begin{equation}
\|e^{\al A^{1/2}}\mu^0\|_{\hat{W}^{\ka_q s, 2q'}} < \frac{\nu^2-2\beta}{2C_q|\M|} \quad \text{and} \quad \|\mu\|_{C_T^0\G_{\phi}^{(\ka_q+\frac{1}{q'}), 2q'}} < \infty,
\end{equation}
then $\|e^{\phi^t A^{1/2}}\mu^t\|_{\hat{W}^{\ka_q s, 2q'}}$ is strictly decreasing on $[0,T]$. A fortiori,
\begin{equation}
\|e^{\phi^t A^{1/2}}\mu^t\|_{\hat{W}^{1-\ga,2q'}} \leq \|e^{\al A^{1/2}}\mu^0\|_{\hat{W}^{\ka_q, 2q'}} \qquad \forall t\in [0,T].
\end{equation}
Therefore, suppose that
\begin{align}
\|e^{\al A^{1/2}} \mu^0 \|_{\hat{W}^{\ka_q s, 2q'}} < \min\left\{\frac{\nu^2-2\beta}{2C_q|\M|}, \frac{\nu^2-2\beta}{2C_{r,q}|\M|}\right\}, \\
\|e^{\al A^{1/2}} \mu^0 \|_{\hat{W}^{\ka_r s, r}} < \frac{\nu^2-2\beta}{2C_{r,q}|\M|} - \|e^{\al A^{1/2}} \mu^0 \|_{\hat{W}^{\ka_q s, 2q'}}.
\end{align}
Under these assumptions, it follows by repeating the continuity argument from above that the quantity $\|e^{\phi^t A^{1/2}}\mu^t\|_{\hat{W}^{\ka_r s, r}}$ is strictly decreasing on $[0,T]$. With this last bit, the proof of \cref{prop:mon} is complete.
\end{proof}

\subsection{Proof of \cref{thm:main}}\label{ssec:globmr}
We now use the local well-posedness established by \cref{prop:lwp} together with the monotonicity of the Gevrey norm established by \cref{prop:mon} in order to show that with high probability, solutions in the class we consider are global. Moreover, their Gevrey norm strictly decreases as time $t\rightarrow \infty$. This then proves \cref{thm:main}. To show the desired result, we use a refined reformulation of the iterative argument from \cite[Section 5]{BNSW2020}.

Let us first present the case $0<\ga\leq 1$, which is simpler due to not needing a two-tiered norm. Fix $\ep>0$ and suppose that $\mu^0\in \G_{\al+\ep}^{\sigma_0,r}$ for $\sigma_0$ above the regularity threshold $\ka_{0,r}$ given by \cref{prop:mon}. Throughout this subsection, we assume that the parameters $d,\ga,r,s,\sigma_0,\al,\be,\nu$ satisfy all the constraints of \cref{thm:main}. We also assume that
\begin{equation}\label{eq:sigma0tz}
\|\mu^0\|_{\G_{\al+\ep}^{\sigma_0,r}} < \frac{\nu^2-2\beta}{C_{mon}|\M|},
\end{equation}
where $C_{mon}=C_r>0$ is the constant from \cref{prop:mon}. Assuming a realization of $W$ from $\Omega_{\al,\be,\nu}$ and given $r\geq 1$ sufficiently small depending on $d,\ga,s$, \cref{prop:lwp} implies that for any $0<\sigma<\frac{2s-1}{s}$, with $1-\ga\leq \sigma s$, sufficiently large depending on $d,\ga,s,r$, there is a maximal solution $\mu$ to the Cauchy problem \eqref{eq:cp} with lifespan $[0,T_{\max,\sigma,\ep})$, such that $\mu$ belongs to $C_T^0\G_{\phi+\ep}^{\sigma,r}$ for any $0\leq T<T_{\max,\sigma,\ep}$. Our main lemma to conclude global existence is the following result relating the lifespan of $\mu^t$ in $\G_{\phi^t+\ep}^{\sigma,r}$ to the lifespan of $\mu$ in the \emph{larger space} $\G_{\phi^t+\ep'}^{\sigma,r}$, for any $\ep'\in [0,\ep)$.

\begin{lemma}\label{lem:lspan}
Let $\mu$ be as above. There exists a constant $C>0$ depending on $d,\ga,r,s,\sigma,\be,\nu$ such that for any $0\leq \ep_2<\ep_1\leq \ep$, the maximal times of existence $T_{\max,\sigma,\ep_1}, T_{\max,\sigma,\ep_2}$ of $\mu^t$ as taking values in $\G_{\phi^t+\ep_1}^{\sigma,r}, \G_{\phi^t+\ep_2}^{\sigma,r}$, respectively, satisfy the inequality
\begin{equation}\label{eq:Tmaxeplb}
T_{\max,\sigma,\ep_2} \geq T_{\max,\sigma,\ep_1} + C(|\M|\|\mu^0\|_{\G_{\al+\ep}^{\sigma_0,r}})^{-\frac{2s}{2s-\sigma s-1}}.
\end{equation}
\end{lemma}
\begin{proof}
Fix $0<\ep_2 <\ep_1\leq \ep$. For any $\sigma'\geq\sigma$, it follows from \cref{lem:Gemb} that $\mu$ is also a solution in $C_T^0\G_{\phi+\ep_2}^{\sigma',r}$ for any $0\leq T<T_{\max,\sigma,\ep_1}$. Furthermore, we have the quantitative bound
\begin{equation}
\|\mu^t\|_{\G_{\phi^t+\ep_2}^{\sigma',r}} \leq \frac{\lceil{\sigma'-\sigma}\rceil !}{(\ep_1-\ep_2)^{\lceil{\sigma'-\sigma}\rceil}} \|\mu^t\|_{\G_{\phi^t+\ep_1}^{\sigma,r}} \qquad \forall \ 0\leq t<T_{\max,\sigma,\ep_1}.
\end{equation}
Choose $\sigma'$ such that
\begin{equation}
\sigma_0 \geq \sigma' > \ka_{0,r},
\end{equation}
where $\ka_{0,r}$ is the regularity threshold of \cref{prop:mon}. Using the assumption \eqref{eq:sigma0tz}, we can then apply \cref{prop:mon} to conclude that the function $t\mapsto \|\mu^t\|_{\G_{\phi^t+\ep_2}^{\sigma',r}}$ is strictly decreasing on the interval $[0,T]$ for any $T<T_{\max,\sigma,\ep_1}$. In particular, since we can ensure that $\sigma\leq \sigma'$, it follows from this monotonicity and \cref{lem:Gemb} that
\begin{equation}
\|\mu\|_{C_T^0\G_{\phi+\ep_2}^{\sigma,r}} \leq \|\mu\|_{C_T^0\G_{\phi+\ep_2}^{\sigma',r}} \leq \|\mu^0\|_{\G_{\al+\ep_2}^{\sigma',r}} \leq \|\mu^0\|_{\G_{\al+\ep}^{\sigma_0,r}} \qquad \forall \ 0\leq T<T_{\max,\sigma,\ep_1}.
\end{equation}
Let $C_{lwp,\sigma}$ be the constant from \cref{prop:lwp}, and choose $T_* < T_{\max,\sigma,\ep_1}$ so that
\begin{equation}
T_{\max,\sigma,\ep_1}-T_* \leq \frac{C_{lwp,\sigma}(|\M| \|\mu^0\|_{\G_{\al+\ep}^{\sigma_0,r}})^{-\frac{2s}{2s-\sigma s-1}}}{2}.
\end{equation}
Thus by relabeling time, we can apply \cref{prop:lwp} once more, but with initial datum $\mu^{T_*}$, to find that $\mu$ belongs to $C_{T_0}^0\G_{\phi+\ep_2}^{\sigma,r}$, where
\begin{align}
T_0 &= T_* + C_{lwp,\sigma}(|\M| \|\mu^0\|_{\G_{\al+\ep_2}^{\sigma,r}})^{-\frac{2s}{2s-\sigma s-1}} \nn\\
&= (T_*-T_{\max,\sigma,\ep_1}) + T_{\max,\sigma,\ep_1} + C_{lwp,\sigma}(|\M| \|\mu^0\|_{\G_{\al+\ep_2}^{\sigma,r}})^{-\frac{2s}{2s-\sigma s-1}} \nn\\
&\geq T_{\max,\sigma,\ep_1} + \frac{C_{lwp,\sigma}(|\M| \|\mu^0\|_{\G_{\al+\ep}^{\sigma_0,r}})^{-\frac{2s}{2s-\sigma s-1}}}{2},
\end{align}
where we have used that $\|\mu^0\|_{\G_{\al+\ep_2}^{\sigma,r}} \leq \|\mu^0\|_{\G_{\al+\ep_2}^{\sigma_0,r}}$. Taking $C=\frac{C_{lwp,\sigma}}{2}$, the preceding inequality is exactly what we need to show.
\end{proof}

\begin{proof}[Proof of \cref{thm:main} for $\ga\leq 1$]
Fix $0<\ep'<\ep$. Let $\sigma,\sigma_0$ be as above. If $T_{\max, \sigma, \ep'}<\infty$, then let $n\in\N$ be such that $n C (|\M|\|\mu^0\|_{\G_{\al+\ep}^{\sigma_0,r}})^{-\frac{2s}{2s-\sigma s-1}}$ satisfies the inequality
\begin{equation}
n C(|\M|\|\mu^0\|_{\G_{\al+\ep}^{\sigma_0,r}})^{-\frac{2s}{2s-\sigma s-1}} > T_{\max,\sigma,\ep'}-T_{\max,\sigma,\ep},
\end{equation}
where $C$ is the same constant as in the inequality \eqref{eq:Tmaxeplb}. We observe from \cref{lem:lspan} that
\begin{align}
T_{\max,\sigma,\ep'} - T_{\max,\sigma,\ep} &= \sum_{j=0}^{n-1} \paren*{T_{\max, \sigma,\ep - \frac{(j+1)(\ep-\ep')}{n}} - T_{\max,\sigma,\ep-\frac{j(\ep-\ep')}{n}}} \nn\\
&\geq \sum_{j=0}^{n-1} C(|\M|\|\mu^0\|_{\G_{\al+\ep}^{\sigma_0,r}})^{-\frac{2s}{2s-\sigma s-1}}\nn\\
&> T_{\max,\sigma,\ep'}-T_{\max,\sigma,\ep},
\end{align}
which is a contradiction. Thus, $T_{\max,\sigma,\ep'}=\infty$.

We have shown that for any $0<\ep'<\ep$ and any $0<\sigma<\frac{2s-1}{s}$ sufficiently large depending on $d,\ga,s,r$, it holds that $\|\mu\|_{C_T^0\G_{\phi+\ep'}^{\sigma,r}}<\infty$ for all $T>0$. Using the arbitrariness of $\ep'$, we see from \cref{lem:Gemb} that for any $T>0$, $\|\mu\|_{C_T^0\G_{\phi+\ep'}^{\sigma_0+\frac{2}{r},r}} < \infty$. Using that
\begin{equation}
\|\mu^0\|_{\G_{\al+\ep'}^{\sigma_0,r}} \leq \|\mu^0\|_{\G_{\al+\ep}^{\sigma_0,r}} < \frac{\nu^2-2\beta}{C_{mon}|\M|}
\end{equation}
by assumption \eqref{eq:sigma0tz}, where $C_{mon}>0$ is the constant from \cref{prop:mon}, we can apply \cref{prop:mon} on the interval $[0,T]$ to obtain that the function $t\mapsto \|\mu^t\|_{\G_{\phi^t+\ep'}^{\sigma_0,r}}$ is strictly decreasing on $[0,T]$. Since $T>0$ was arbitrary, we see that this monotonicity property holds on the entire interval $[0,\infty)$.

Finally, we show that $\mu$ actually belongs to $C_\infty^0\G_{\phi+\ep}^{\sigma_0,r}$ and that the decreasing property holds on $[0,\infty)$. Note that there is no longer a loss in the Gevrey index value (i.e. $\ep'=\ep$). To this end, we observe from the result of the preceding paragraph and Fatou's lemma that for any $t\geq 0$,
\begin{equation}
\|\mu^t\|_{\G_{\phi^t+\ep}^{\sigma_0,r}} = \lim_{\ep'\rightarrow\ep^-} \|\mu^t\|_{\G_{\phi^t+\ep'}^{\sigma_0,r}} \leq \lim_{\ep'\rightarrow \ep^-} \|\mu^0\|_{\G_{\al+\ep'}^{\sigma_0,r}} \leq \|\mu^0\|_{\G_{\al+\ep}^{\sigma_0,r}} < \infty.
\end{equation}
Similarly, for any $t_2\geq t_1\geq 0$,
\begin{equation}
\|\mu^{t_2}\|_{\G_{\phi^{t_2}+\ep}^{\sigma_0,r}} = \lim_{\ep'\rightarrow \ep^-} \|\mu^{t_2} \|_{\G_{\phi^{t_2}+\ep'}^{\sigma_0,r}}\leq \lim_{\ep'\rightarrow \ep^-} \|\mu^{t_1} \|_{\G_{\phi^{t_1}+\ep'}^{\sigma_0,r}} = \|\mu^{t_1} \|_{\G_{\phi^{t_1}+\ep}^{\sigma_0,r}},
\end{equation}
where the inequality is strict if $t_2>t_1$. This completes the proof of \cref{thm:main} in the case $\ga\leq 1$.
\end{proof}

\medskip
Let us now present the case $\ga>1$, the strategy for which is similar to before. Fix $\ep>0$ and suppose that $\mu^0\in \X_{\al+\ep,\ga}^{\sigma_{0,r},\sigma_{0,q},r,\frac{2q}{q-1}}$ for $\sigma_{0,r},\sigma_{0,q}$ above the regularity thresholds $\ka_{0,r},\ka_{0,q}$ given by \cref{prop:mon} and $1<q<\frac{d}{\ga-1}$. Let us drop the subscript $\ga$ from the notation $\X_{\al+\ep,\ga}$, as $\ga>1$ is fixed. In what follows, we assume that the parameters $d,\ga,r,q,s,\sigma_{0,r},\sigma_{0,q},\al,\be,\nu$ satisfy all the constraints of \cref{thm:main}. We also assume that
\begin{equation}\label{eq:qsigma0tz}
\|\mu^0\|_{\X_{\al+\ep}^{\sigma_{0,r},\sigma_{0,q}, r,\frac{2q}{q-1}}} < \frac{\nu^2-2\beta}{\max\{C_{mon,r},C_{mon,q}\}|\M|},
\end{equation}
where $C_{mon,r}=C_r,C_{mon,q}=C_q>0$ are the constants from \cref{prop:mon}. Assuming a realization of $W$ from $\Omega_{\al,\be,\nu}$, \cref{prop:lwp} implies that given $r\geq 1$ sufficiently small depending on $d,s,\ga$ and $0<\sigma<\frac{2s-1}{s}$ sufficiently large depending on $d,\ga,s,r$, there is a maximal solution $\mu$ to the Cauchy problem \eqref{eq:cp} with lifespan $[0,T_{\max,\sigma,\ep})$, such that $\mu$ belongs to $C_T^0\X_{\phi+\ep}^{\sigma,0,r,\frac{2q}{q-1}}$ for any $0\leq T<T_{\max,\sigma,\ep}$. Here, we recycle the notation $T_{\max,\ep,\sigma}$ used above. Analogous to \cref{lem:lspan}, we have the following relation for the maximal lifespans as we decrease $\ep$.

\begin{lemma}\label{lem:qlspan}
Let $\mu$ be as above. There exists a constant $C>0$ depending on $d,\ga,r,q,s,\sigma,\be,\nu$ such that for any $0\leq \ep_2<\ep_1\leq \ep$, the maximal times of existence $T_{\max,\sigma,\ep_1}, T_{\max,\sigma,\ep_2}$ of $\mu^t$ as taking values in $\X_{\phi^t+\ep_1}^{\sigma,0,r,\frac{2q}{q-1}}, \X_{\phi^t+\ep_2}^{\sigma,0,r,\frac{2q}{q-1}}$, respectively, satisfy the inequality
\begin{equation}\label{eq:qTmaxeplb}
T_{\max,\sigma,\ep_2} \geq T_{\max,\sigma,\ep_1} + C(|\M|\|\mu^0\|_{\X_{\al+\ep}^{\sigma_{0,r},\sigma_{0,q},r,\frac{2q}{q-1}}})^{-\frac{2s}{2s-\sigma s-1}}.
\end{equation}
\end{lemma}
\begin{proof}
Fix $0<\ep_2 <\ep_1\leq \ep$. For any $\sigma_r \geq\sigma$ and $\sigma_q \geq 0$, it follows from \cref{lem:Gemb} that $\mu$ is also a solution in $C_T^0\X_{\phi+\ep_2}^{\sigma_r,\sigma_q,r,\frac{2q}{q-1}}$ for any $0\leq T<T_{\max,\sigma,\ep_1}$. Furthermore, we have the quantitative bounds
\begin{align}
\|\mu^t\|_{\G_{\phi^t+\ep_2}^{\sigma_r,r}} \leq \frac{\lceil{\sigma_r-\sigma}\rceil !}{(\ep_1-\ep_2)^{\lceil{\sigma_r-\sigma}\rceil}} \|\mu^t\|_{\G_{\phi^t+\ep_1}^{\sigma_r,r}} \qquad \forall \ 0\leq t<T_{\max,\sigma,\ep_1},\\
\|\mu^t\|_{\G_{\phi^t+\ep_2}^{\sigma_q,\frac{2q}{q-1}}} \leq \frac{\lceil{\sigma_q}\rceil !}{(\ep_1-\ep_2)^{\lceil{\sigma_q}\rceil}} \|\mu^t\|_{\G_{\phi^t+\ep_1}^{0,\frac{2q}{q-1}}} \qquad \forall \ 0\leq t<T_{\max,\sigma,\ep_1},
\end{align}
which of course imply that $\|\mu\|_{C_T^0\X_{\phi+\ep_2}^{\sigma_r,\sigma_q,r,\frac{2q}{q-1}}}$ is finite on any compact subinterval of $[0,T_{\max,\sigma,\ep_1})$. Choose $\sigma_r,\sigma_q$ such that
\begin{equation}
\sigma_{0,r} \geq \sigma_r > \ka_{0,r} \quad \text{and} \quad \sigma_{0,q} \geq \sigma_q > \ka_{0,q},
\end{equation}
where $\ka_{0,r},\ka_{0,q}$ are the regularity thresholds of \cref{prop:mon}. Using the assumption \eqref{eq:sigma0tz}, we can then apply \cref{prop:mon} to conclude that the function $t\mapsto \|\mu^t\|_{\X_{\phi^t+\ep_2}^{\sigma_r,\sigma_q,r,\frac{2q}{q-1}}}$ is strictly decreasing on the interval $[0,T]$ for any $T<T_{\max,\sigma,\ep_1}$. In particular, since we can ensure $\sigma\leq \sigma_r$ and $0\leq \sigma_q$, it follows from this monotonicity and \cref{lem:Gemb} that
\begin{equation}
\|\mu\|_{C_T^0\X_{\phi+\ep_2}^{\sigma,0,r,\frac{2q}{q-1}}} \leq \|\mu\|_{C_T^0\X_{\phi+\ep_2}^{\sigma_r,\sigma_q,r,\frac{2q}{q-1}}} \leq \|\mu^0\|_{\X_{\al+\ep_2}^{\sigma_r,\sigma_q,r,\frac{2q}{q-1}}} \leq \|\mu^0\|_{\X_{\al+\ep}^{\sigma_{0,r},\sigma_{0,q},r,\frac{2q}{q-1}}}
\end{equation}
for all $0\leq T<T_{\max,\sigma,\ep_1}$. Let $C_{lwp,\sigma}$ be the constant from \cref{prop:lwp}, and choose $T_* < T_{\max,\sigma,\ep_1}$ so that
\begin{equation}
T_{\max,\sigma,\ep_1}-T_* \leq \frac{C_{lwp,\sigma}(|\M| \|\mu^0\|_{\X_{\al+\ep}^{\sigma_{0,r},\sigma_{0,q},r,\frac{2q}{q-1}}})^{-\frac{2s}{2s-\sigma s-1}}}{2}.
\end{equation}
Then by the same argument as in the proof of \cref{lem:lspan}, we see that $\mu \in C_{T_0}^0\X_{\phi+\ep_2}^{\sigma,0,r,\frac{2q}{q-1}}$, where
\begin{equation}
T_0 \geq T_{\max,\sigma,\ep_1} + \frac{C_{lwp,\sigma}(|\M| \|\mu^0\|_{\X_{\al+\ep}^{\sigma_{0,r},\sigma_{0,q},r,\frac{2q}{q-1}}})^{-\frac{2s}{2s-\sigma s-1}}}{2},
\end{equation}
which completes the proof of the lemma.
\end{proof}

\begin{proof}[Proof of \cref{thm:main} for $\ga>1$]
Fix $0<\ep'<\ep$. Let $\sigma,\sigma_{0,r},\sigma_{0,q}$ be as above. If $T_{\max, \sigma, \ep'}<\infty$, then choosing $n\in\N$ such that
\begin{equation}
n C(|\M|\|\mu^0\|_{\X_{\al+\ep}^{\sigma_{0,r},\sigma_{0,q},r,\frac{2q}{q-1}}})^{-\frac{2s}{2s-\sigma s-1}} > T_{\max,\sigma,\ep'}-T_{\max,\sigma,\ep},
\end{equation}
where $C$ is the same constant as in the inequality \eqref{eq:qTmaxeplb}, one can use the same argument from above, except now invoking \cref{lem:qlspan}, to derive a contradiction. Thus, $T_{\max,\sigma,\ep'}=\infty$.

We have shown that for any $0<\ep'<\ep$ and all $0<\sigma<\frac{2s-1}{s}$ sufficiently large depending on $d,s,\ga,r$, it holds that $\|\mu\|_{C_T^0\X_{\phi+\ep'}^{\sigma,0,r,\frac{2q}{q-1}}}<\infty$ for all $T>0$. By the arbitrariness of $\ep'$,  \cref{lem:Gemb} implies for any $T>0$, $\|\mu\|_{C_T^0\X_{\phi+\ep'}^{\sigma_{0,r}+\frac{2}{r}, \sigma_{0,q}+\frac{q-1}{q},r,\frac{2q}{q-1}}} < \infty$. Using initial datum assumption \eqref{eq:qsigma0tz}, we can apply \cref{prop:mon} on the interval $[0,T]$ to obtain that the function $t\mapsto \|\mu^t\|_{\X_{\phi^t+\ep'}^{\sigma_{0,r},\sigma_{0,q},r,\frac{2q}{q-1}}}$ is strictly decreasing on $[0,T]$. Since $T>0$ was arbitrary, we see that this monotonicity property holds on the entire interval $[0,\infty)$. Using the same Fatou argument from before, we complete the proof in the $\ga>1$ case.
\end{proof}

\bibliographystyle{alpha}
\bibliography{RD}

\newcommand{\etalchar}[1]{$^{#1}$}
\begin{thebibliography}{MEKBS03}

\bibitem[AMS11]{AMS2011}
Luigi Ambrosio, Edoardo Mainini, and Sylvia Serfaty.
\newblock Gradient flow of the {C}hapman-{R}ubinstein-{S}chatzman model for
  signed vortices.
\newblock {\em Ann. Inst. H. Poincar\'{e} Anal. Non Lin\'{e}aire},
  28(2):217--246, 2011.

\bibitem[AS08]{AS2008}
Luigi Ambrosio and Sylvia Serfaty.
\newblock A gradient flow approach to an evolution problem arising in
  superconductivity.
\newblock {\em Comm. Pure Appl. Math.}, 61(11):1495--1539, 2008.

\bibitem[BCCP98]{BCCP1998}
D.~Benedetto, E.~Caglioti, J.~A. Carrillo, and M.~Pulvirenti.
\newblock A non-{M}axwellian steady distribution for one-dimensional granular
  media.
\newblock {\em J. Statist. Phys.}, 91(5-6):979--990, 1998.

\bibitem[BCM08]{BCM2008}
Adrien Blanchet, Jos\'{e}~A. Carrillo, and Nader Masmoudi.
\newblock Infinite time aggregation for the critical {P}atlak-{K}eller-{S}egel
  model in {$\Bbb R^2$}.
\newblock {\em Comm. Pure Appl. Math.}, 61(10):1449--1481, 2008.

\bibitem[BCP97]{BCP1997}
D.~Benedetto, E.~Caglioti, and M.~Pulvirenti.
\newblock A kinetic equation for granular media.
\newblock {\em RAIRO Mod\'{e}l. Math. Anal. Num\'{e}r.}, 31(5):615--641, 1997.

\bibitem[BDP06]{BDP2006}
Adrien Blanchet, Jean Dolbeault, and Beno\^{\i}t Perthame.
\newblock Two-dimensional {K}eller-{S}egel model: optimal critical mass and
  qualitative properties of the solutions.
\newblock {\em Electron. J. Differential Equations}, pages No. 44, 32, 2006.

\bibitem[BFGM19]{BFGM2019}
Lisa Beck, Franco Flandoli, Massimiliano Gubinelli, and Mario Maurelli.
\newblock Stochastic {ODE}s and stochastic linear {PDE}s with critical drift:
  regularity, duality and uniqueness.
\newblock {\em Electron. J. Probab.}, 24:Paper No. 136, 72, 2019.

\bibitem[BH17]{BH2017}
Jacob Bedrossian and Siming He.
\newblock Suppression of blow-up in {P}atlak-{K}eller-{S}egel via shear flows.
\newblock {\em SIAM J. Math. Anal.}, 49(6):4722--4766, 2017.

\bibitem[BIK15]{BIK2015}
Piotr Biler, Cyril Imbert, and Grzegorz Karch.
\newblock The nonlocal porous medium equation: {B}arenblatt profiles and other
  weak solutions.
\newblock {\em Arch. Ration. Mech. Anal.}, 215(2):497--529, 2015.

\bibitem[BJW19]{BJW2019edp}
Didier Bresch, Pierre-Emmanuel Jabin, and Zhenfu Wang.
\newblock Modulated free energy and mean field limit.
\newblock {\em S{\'e}minaire Laurent Schwartz--EDP et applications}, pages
  1--22, 2019.

\bibitem[BKM10]{BKM2010}
Piotr Biler, Grzegorz Karch, and R\'{e}gis Monneau.
\newblock Nonlinear diffusion of dislocation density and self-similar
  solutions.
\newblock {\em Comm. Math. Phys.}, 294(1):145--168, 2010.

\bibitem[BLL12]{BLL2012}
Andrea~L. Bertozzi, Thomas Laurent, and Flavien L\'{e}ger.
\newblock Aggregation and spreading via the {N}ewtonian potential: the dynamics
  of patch solutions.
\newblock {\em Math. Models Methods Appl. Sci.}, 22(suppl. 1):1140005, 39,
  2012.

\bibitem[BLR11]{BLR2011}
Andrea~L. Bertozzi, Thomas Laurent, and Jes\'{u}s Rosado.
\newblock {$L^p$} theory for the multidimensional aggregation equation.
\newblock {\em Comm. Pure Appl. Math.}, 64(1):45--83, 2011.

\bibitem[BNSW20]{BNSW2020}
Tristan Buckmaster, Andrea Nahmod, Gigliola Staffilani, and Klaus Widmayer.
\newblock The surface quasi-geostrophic equation with random diffusion.
\newblock {\em Int. Math. Res. Not. IMRN}, (23):9370--9385, 2020.

\bibitem[BO19]{BO2019}
Robert~J. Berman and Magnus \"{O}nnheim.
\newblock Propagation of chaos for a class of first order models with singular
  mean field interactions.
\newblock {\em SIAM J. Math. Anal.}, 51(1):159--196, 2019.

\bibitem[BSV19]{BSV2019}
Tristan Buckmaster, Steve Shkoller, and Vlad Vicol.
\newblock Nonuniqueness of weak solutions to the {SQG} equation.
\newblock {\em Comm. Pure Appl. Math.}, 72(9):1809--1874, 2019.

\bibitem[BvCK20]{BCCK2020}
Tristan Buckmaster, Sun\v{c}ica \v{C}ani\'{c}, Peter Constantin, and
  Alexander~A. Kiselev.
\newblock {\em Progress in mathematical fluid dynamics}, volume 2272 of {\em
  Lecture Notes in Mathematics}.
\newblock Springer, Cham; Centro Internazionale Matematico Estivo (C.I.M.E.),
  Florence, [2020] \copyright 2020.
\newblock Fondazione CIME/CIME Foundation Subseries.

\bibitem[CCC{\etalchar{+}}12]{CCCGW2012}
Dongho Chae, Peter Constantin, Diego C\'{o}rdoba, Francisco Gancedo, and
  Jiahong Wu.
\newblock Generalized surface quasi-geostrophic equations with singular
  velocities.
\newblock {\em Comm. Pure Appl. Math.}, 65(8):1037--1066, 2012.

\bibitem[CCH14]{CCH2014}
Jos\'{e}~Antonio Carrillo, Young-Pil Choi, and Maxime Hauray.
\newblock The derivation of swarming models: mean-field limit and {W}asserstein
  distances.
\newblock In {\em Collective dynamics from bacteria to crowds}, volume 553 of
  {\em CISM Courses and Lect.}, pages 1--46. Springer, Vienna, 2014.

\bibitem[CD14]{CD2014}
Juan~F. Campos and Jean Dolbeault.
\newblock Asymptotic estimates for the parabolic-elliptic {K}eller-{S}egel
  model in the plane.
\newblock {\em Comm. Partial Differential Equations}, 39(5):806--841, 2014.

\bibitem[CF02]{CF2002}
Diego Cordoba and Charles Fefferman.
\newblock Growth of solutions for {QG} and 2{D} {E}uler equations.
\newblock {\em J. Amer. Math. Soc.}, 15(3):665--670, 2002.

\bibitem[CG15]{CG2015}
K.~Chouk and M.~Gubinelli.
\newblock Nonlinear {PDE}s with modulated dispersion {I}: {N}onlinear
  {S}chr\"{o}dinger equations.
\newblock {\em Comm. Partial Differential Equations}, 40(11):2047--2081, 2015.

\bibitem[CGSI19]{CGI2019}
Diego C\'{o}rdoba, Javier G\'{o}mez-Serrano, and Alexandru~D. Ionescu.
\newblock Global solutions for the generalized {SQG} patch equation.
\newblock {\em Arch. Ration. Mech. Anal.}, 233(3):1211--1251, 2019.

\bibitem[CHSV15]{CHSV2015}
J.~A. Carrillo, Y.~Huang, M.~C. Santos, and J.~L. V\'{a}zquez.
\newblock Exponential convergence towards stationary states for the 1{D} porous
  medium equation with fractional pressure.
\newblock {\em J. Differential Equations}, 258(3):736--763, 2015.

\bibitem[CJ21]{CJ2021}
Young-Pil Choi and In-Jee Jeong.
\newblock Classical solutions for fractional porous medium flow.
\newblock {\em Nonlinear Anal.}, 210:Paper No. 112393, 13, 2021.

\bibitem[CMT94]{CMT1994}
Peter Constantin, Andrew~J Majda, and Esteban Tabak.
\newblock Formation of strong fronts in the {2-D} quasigeostrophic thermal
  active scalar.
\newblock {\em Nonlinearity}, 7(6):1495, 1994.

\bibitem[CMV06]{CMV2006}
Jos\'{e}~A. Carrillo, Robert~J. McCann, and C\'{e}dric Villani.
\newblock Contractions in the 2-{W}asserstein length space and thermalization
  of granular media.
\newblock {\em Arch. Ration. Mech. Anal.}, 179(2):217--263, 2006.

\bibitem[CRS96]{CRS1996}
S.~J. Chapman, J.~Rubinstein, and M.~Schatzman.
\newblock A mean-field model of superconducting vortices.
\newblock {\em European J. Appl. Math.}, 7(2):97--111, 1996.

\bibitem[CSV13]{CSV2013}
Luis Caffarelli, Fernando Soria, and Juan~Luis V\'{a}zquez.
\newblock Regularity of solutions of the fractional porous medium flow.
\newblock {\em J. Eur. Math. Soc. (JEMS)}, 15(5):1701--1746, 2013.

\bibitem[CV10]{CV2010}
Luis Caffarelli and Alexis Vasseur.
\newblock Drift diffusion equations with fractional diffusion and the
  quasi-geostrophic equation.
\newblock {\em Annals of Mathematics}, 171(3):1903--1930, 2010.

\bibitem[CV11a]{CV2011}
Luis Caffarelli and Juan~Luis Vazquez.
\newblock Nonlinear porous medium flow with fractional potential pressure.
\newblock {\em Arch. Ration. Mech. Anal.}, 202(2):537--565, 2011.

\bibitem[CV11b]{CV2011asy}
Luis~A. Caffarelli and Juan~Luis V\'{a}zquez.
\newblock Asymptotic behaviour of a porous medium equation with fractional
  diffusion.
\newblock {\em Discrete Contin. Dyn. Syst.}, 29(4):1393--1404, 2011.

\bibitem[CV12]{CV2012nmp}
Peter Constantin and Vlad Vicol.
\newblock Nonlinear maximum principles for dissipative linear nonlocal
  operators and applications.
\newblock {\em Geometric and Functional Analysis}, 22(5):1289--1321, Oct 2012.

\bibitem[CV15]{CV2015}
L.~Caffarelli and J.~L. V\'{a}zquez.
\newblock Regularity of solutions of the fractional porous medium flow with
  exponent 1/2.
\newblock {\em Algebra i Analiz}, 27(3):125--156, 2015.

\bibitem[CW99]{CW1999QG}
P.~Constantin and J.~Wu.
\newblock Behavior of solutions of {2D} quasi-geostrophic equations.
\newblock {\em SIAM Journal on Mathematical Analysis}, 30(5):937--948, 1999.

\bibitem[dBD02]{dBD2002}
A.~de~Bouard and A.~Debussche.
\newblock Finite-time blow-up in the additive supercritical stochastic
  nonlinear {S}chr\"{o}dinger equation: the real noise case.
\newblock In {\em The legacy of the inverse scattering transform in applied
  mathematics ({S}outh {H}adley, {MA}, 2001)}, volume 301 of {\em Contemp.
  Math.}, pages 183--194. Amer. Math. Soc., Providence, RI, 2002.

\bibitem[dBD05]{dBD2005}
Anne de~Bouard and Arnaud Debussche.
\newblock Blow-up for the stochastic nonlinear {S}chr\"{o}dinger equation with
  multiplicative noise.
\newblock {\em Ann. Probab.}, 33(3):1078--1110, 2005.

\bibitem[DL89a]{DL1989b}
R.~J. DiPerna and P.-L. Lions.
\newblock On the {C}auchy problem for {B}oltzmann equations: global existence
  and weak stability.
\newblock {\em Ann. of Math. (2)}, 130(2):321--366, 1989.

\bibitem[DL89b]{DL1989ode}
R.~J. DiPerna and P.-L. Lions.
\newblock Ordinary differential equations, transport theory and {S}obolev
  spaces.
\newblock {\em Invent. Math.}, 98(3):511--547, 1989.

\bibitem[Dob79]{Dobrushin1979}
R.~L. Dobru\v{s}in.
\newblock Vlasov equations.
\newblock {\em Funktsional. Anal. i Prilozhen.}, 13(2):48--58, 96, 1979.

\bibitem[DT11]{DT2011}
Arnaud Debussche and Yoshio Tsutsumi.
\newblock 1{D} quintic nonlinear {S}chr\"{o}dinger equation with white noise
  dispersion.
\newblock {\em J. Math. Pures Appl. (9)}, 96(4):363--376, 2011.

\bibitem[Due16]{Duerinckx2016}
M.~Duerinckx.
\newblock Mean-field limits for some {Riesz} interaction gradient flows.
\newblock {\em SIAM Journal on Mathematical Analysis}, 48(3):2269--2300, 2016.

\bibitem[E94]{E1994}
Weinan E.
\newblock Dynamics of vortices in {G}inzburg-{L}andau theories with
  applications to superconductivity.
\newblock {\em Phys. D}, 77(4):383--404, 1994.

\bibitem[FGL21]{FGL2021}
Franco Flandoli, Lucio Galeati, and Dejun Luo.
\newblock Delayed blow-up by transport noise.
\newblock {\em Comm. Partial Differential Equations}, 46(9):1757--1788, 2021.

\bibitem[FGP10]{FGP2010}
F.~Flandoli, M.~Gubinelli, and E.~Priola.
\newblock Well-posedness of the transport equation by stochastic perturbation.
\newblock {\em Invent. Math.}, 180(1):1--53, 2010.

\bibitem[FGP11]{FGP2011}
F.~Flandoli, M.~Gubinelli, and E.~Priola.
\newblock Full well-posedness of point vortex dynamics corresponding to
  stochastic 2{D} {E}uler equations.
\newblock {\em Stochastic Process. Appl.}, 121(7):1445--1463, 2011.

\bibitem[FL21]{FL2021}
Franco Flandoli and Dejun Luo.
\newblock High mode transport noise improves vorticity blow-up control in 3{D}
  {N}avier-{S}tokes equations.
\newblock {\em Probab. Theory Related Fields}, 180(1-2):309--363, 2021.

\bibitem[Fla11]{Flandoli2011}
Franco Flandoli.
\newblock {\em Random perturbation of {PDE}s and fluid dynamic models}, volume
  2015 of {\em Lecture Notes in Mathematics}.
\newblock Springer, Heidelberg, 2011.
\newblock Lectures from the 40th Probability Summer School held in Saint-Flour,
  2010, \'{E}cole d'\'{E}t\'{e} de Probabilit\'{e}s de Saint-Flour.
  [Saint-Flour Probability Summer School].

\bibitem[FT89]{FT1989}
C.~Foias and R.~Temam.
\newblock Gevrey class regularity for the solutions of the {N}avier-{S}tokes
  equations.
\newblock {\em J. Funct. Anal.}, 87(2):359--369, 1989.

\bibitem[Gan08]{Gancedo2008}
Francisco Gancedo.
\newblock Existence for the {$\alpha$}-patch model and the {QG} sharp front in
  {S}obolev spaces.
\newblock {\em Adv. Math.}, 217(6):2569--2598, 2008.

\bibitem[GHV14]{GhV2014}
Nathan~E. Glatt-Holtz and Vlad~C. Vicol.
\newblock Local and global existence of smooth solutions for the stochastic
  {E}uler equations with multiplicative noise.
\newblock {\em Ann. Probab.}, 42(1):80--145, 2014.

\bibitem[GL97]{GL1997}
Giambattista Giacomin and Joel~L. Lebowitz.
\newblock Phase segregation dynamics in particle systems with long range
  interactions. {I}. {M}acroscopic limits.
\newblock {\em J. Statist. Phys.}, 87(1-2):37--61, 1997.

\bibitem[GL98]{GL1998}
Giambattista Giacomin and Joel~L. Lebowitz.
\newblock Phase segregation dynamics in particle systems with long range
  interactions. {II}. {I}nterface motion.
\newblock {\em SIAM J. Appl. Math.}, 58(6):1707--1729, 1998.

\bibitem[GLM00]{GLM2000}
Giambattista Giacomin, Joel~L. Lebowitz, and Rossana Marra.
\newblock Macroscopic evolution of particle systems with short- and long-range
  interactions.
\newblock {\em Nonlinearity}, 13(6):2143--2162, 2000.

\bibitem[GM18]{GM2018}
Tej-Eddine Ghoul and Nader Masmoudi.
\newblock Minimal mass blowup solutions for the {P}atlak-{K}eller-{S}egel
  equation.
\newblock {\em Comm. Pure Appl. Math.}, 71(10):1957--2015, 2018.

\bibitem[Gol16]{Golse2016ln}
Fran\c{c}ois Golse.
\newblock On the dynamics of large particle systems in the mean field limit.
\newblock In {\em Macroscopic and large scale phenomena: coarse graining, mean
  field limits and ergodicity}, volume~3 of {\em Lect. Notes Appl. Math.
  Mech.}, pages 1--144. Springer, 2016.

\bibitem[GP03]{GP2003}
Veysel Gazi and Kevin~M. Passino.
\newblock Stability analysis of swarms.
\newblock {\em IEEE Trans. Automat. Control}, 48(4):692--697, 2003.

\bibitem[GP21]{GP2021gsqg}
Francisco Gancedo and Neel Patel.
\newblock On the local existence and blow-up for generalized {SQG} patches.
\newblock {\em Ann. PDE}, 7(1):Paper No. 4, 63, 2021.

\bibitem[H\"33]{Holder1933}
Ernst H\"{o}lder.
\newblock \"{U}ber die unbeschr\"{a}nkte {F}ortsetzbarkeit einer stetigen
  ebenen {B}ewegung in einer unbegrenzten inkompressiblen {F}l\"{u}ssigkeit.
\newblock {\em Math. Z.}, 37(1):727--738, 1933.

\bibitem[Hau09]{Hauray2009}
Maxime Hauray.
\newblock Wasserstein distances for vortices approximation of {E}uler-type
  equations.
\newblock {\em Math. Models Methods Appl. Sci.}, 19(8):1357--1384, 2009.

\bibitem[HK21]{HK2021}
Siming He and Alexander Kiselev.
\newblock Small-scale creation for solutions of the {SQG} equation.
\newblock {\em Duke Math. J.}, 170(5):1027--1041, 2021.

\bibitem[HP06]{HP2006}
Darryl~D. Holm and Vakhtang Putkaradze.
\newblock Formation of clumps and patches in self-aggregation of finite-size
  particles.
\newblock {\em Phys. D}, 220(2):183--196, 2006.

\bibitem[HPGS95]{HPGS1995}
Isaac~M. Held, Raymond~T. Pierrehumbert, Stephen~T. Garner, and Kyle~L.
  Swanson.
\newblock Surface quasi-geostrophic dynamics.
\newblock {\em J. Fluid Mech.}, 282:1--20, 1995.

\bibitem[IXZ21]{IXZ2021}
Gautam Iyer, Xiaoqian Xu, and Andrej Zlato\v{s}.
\newblock Convection-induced singularity suppression in the {K}eller-{S}egel
  and other non-linear {PDE}s.
\newblock {\em Trans. Amer. Math. Soc.}, 374(9):6039--6058, 2021.

\bibitem[Jab14]{Jab2014}
Pierre-Emmanuel Jabin.
\newblock A review of the mean field limits for {V}lasov equations.
\newblock {\em Kinet. Relat. Models}, 7(4):661--711, 2014.

\bibitem[JL92]{JL1992}
W.~J\"{a}ger and S.~Luckhaus.
\newblock On explosions of solutions to a system of partial differential
  equations modelling chemotaxis.
\newblock {\em Trans. Amer. Math. Soc.}, 329(2):819--824, 1992.

\bibitem[JW18]{JW2018}
Pierre-Emmanuel Jabin and Zhenfu Wang.
\newblock Quantitative estimates of propagation of chaos for stochastic systems
  with {$W^{-1,\infty}$} kernels.
\newblock {\em Invent. Math.}, 214(1):523--591, 2018.

\bibitem[KNV07]{KNV2007gwp}
A.~Kiselev, F.~Nazarov, and A.~Volberg.
\newblock Global well-posedness for the critical {2D} dissipative
  quasi-geostrophic equation.
\newblock {\em Inventiones mathematicae}, 167(3):445--453, Mar 2007.

\bibitem[KS70]{KS1970}
Evelyn~F. Keller and Lee~A. Segel.
\newblock Initiation of slime mold aggregation viewed as an instability.
\newblock {\em Journal of Theoretical Biology}, 26(3):399--415, 1970.

\bibitem[KX16]{KX2016}
Alexander Kiselev and Xiaoqian Xu.
\newblock Suppression of chemotactic explosion by mixing.
\newblock {\em Arch. Ration. Mech. Anal.}, 222(2):1077--1112, 2016.

\bibitem[LMS18]{LMS2018}
Stefano Lisini, Edoardo Mainini, and Antonio Segatti.
\newblock A gradient flow approach to the porous medium equation with
  fractional pressure.
\newblock {\em Arch. Ration. Mech. Anal.}, 227(2):567--606, 2018.

\bibitem[LZ00]{LZ2000}
Fanghua Lin and Ping Zhang.
\newblock On the hydrodynamic limit of {G}inzburg-{L}andau vortices.
\newblock {\em Discrete Contin. Dynam. Systems}, 6(1):121--142, 2000.

\bibitem[Mai12]{Mainini2012}
Edoardo Mainini.
\newblock Well-posedness for a mean field model of {G}inzburg-{L}andau vortices
  with opposite degrees.
\newblock {\em NoDEA Nonlinear Differential Equations Appl.}, 19(2):133--158,
  2012.

\bibitem[MB02]{MB2002}
Andrew~J. Majda and Andrea~L. Bertozzi.
\newblock {\em Vorticity and incompressible flow}, volume~27 of {\em Cambridge
  Texts in Applied Mathematics}.
\newblock Cambridge University Press, Cambridge, 2002.

\bibitem[MEK99]{MEk1999}
Alexander Mogilner and Leah Edelstein-Keshet.
\newblock A non-local model for a swarm.
\newblock {\em J. Math. Biol.}, 38(6):534--570, 1999.

\bibitem[MEKBS03]{MEkBS2003}
A.~Mogilner, L.~Edelstein-Keshet, L.~Bent, and A.~Spiros.
\newblock Mutual interactions, potentials, and individual distance in a social
  aggregation.
\newblock {\em J. Math. Biol.}, 47(4):353--389, 2003.

\bibitem[MP12]{MP2012book}
Carlo Marchioro and Mario Pulvirenti.
\newblock {\em Mathematical theory of incompressible nonviscous fluids},
  volume~96.
\newblock Springer Science \& Business Media, 2012.

\bibitem[MST21]{MST2021}
Oleksandr Misiats, Oleksandr Stanzhytskyi, and Ihsan Topaloglu.
\newblock On global existence and blowup of solutions of stochastic
  {Keller-Segel} type equation.
\newblock {\em arXiv preprint arXiv:2107.12419}, 2021.

\bibitem[MZ05]{MZ2005}
Nader Masmoudi and Ping Zhang.
\newblock Global solutions to vortex density equations arising from
  sup-conductivity.
\newblock {\em Ann. Inst. H. Poincar\'{e} Anal. Non Lin\'{e}aire},
  22(4):441--458, 2005.

\bibitem[Nan73]{Nanjundiah1973}
Vidyanand Nanjundiah.
\newblock Chemotaxis, signal relaying and aggregation morphology.
\newblock {\em Journal of Theoretical Biology}, 42(1):63--105, 1973.

\bibitem[NPS01]{NPS2001}
Juan Nieto, Fr\'{e}d\'{e}ric Poupaud, and Juan Soler.
\newblock High-field limit for the {V}lasov-{P}oisson-{F}okker-{P}lanck system.
\newblock {\em Arch. Ration. Mech. Anal.}, 158(1):29--59, 2001.

\bibitem[NRS21]{NRS2021}
Quoc~Hung Nguyen, Matthew Rosenzweig, and Sylvia Serfaty.
\newblock Mean-field limits of {Riesz-typ}e singular flows with possible
  multiplicative transport noise.
\newblock {\em arXiv preprint arXiv:2107.02592}, 2021.

\bibitem[Pat53]{Patlak1953}
Clifford~S. Patlak.
\newblock Random walk with persistence and external bias.
\newblock {\em The bulletin of mathematical biophysics}, 15(3):311--338, 1953.

\bibitem[Ped13]{Pedlosky2013}
Joseph Pedlosky.
\newblock {\em Geophysical fluid dynamics}.
\newblock Springer Science \& Business Media, 2013.

\bibitem[PHS94]{PHS1994spec}
Raymond~T. Pierrehumbert, Isaac~M. Held, and Kyle~L. Swanson.
\newblock Spectra of local and nonlocal two-dimensional turbulence.
\newblock {\em Chaos, Solitons \& Fractals}, 4(6):1111 -- 1116, 1994.
\newblock Special Issue: Chaos Applied to Fluid Mixing.

\bibitem[Pou02]{Poupaud2002}
Fr\'{e}d\'{e}ric Poupaud.
\newblock Diagonal defect measures, adhesion dynamics and {E}uler equation.
\newblock {\em Methods Appl. Anal.}, 9(4):533--561, 2002.

\bibitem[Res92]{Resnick1992}
Sidney Resnick.
\newblock {\em Adventures in stochastic processes}.
\newblock Birkh\"{a}user Boston, Inc., Boston, MA, 1992.

\bibitem[Res95]{Resnick1995}
Serge Resnick.
\newblock {\em Dyanmical Problems in Non-Linear Advective Partial Differential
  Equations}.
\newblock PhD thesis, University of Chicago, 1995.

\bibitem[Ser20]{Serfaty2020}
Sylvia Serfaty.
\newblock Mean field limit for {Coulomb-type} flows.
\newblock {\em Duke Math. J.}, 169(15):2887--2935, 10 2020.
\newblock Appendix with Mitia Duerinckx.

\bibitem[SV14]{SV2014}
Sylvia Serfaty and Juan~Luis V\'{a}zquez.
\newblock A mean field equation as limit of nonlinear diffusions with
  fractional {L}aplacian operators.
\newblock {\em Calc. Var. Partial Differential Equations}, 49(3-4):1091--1120,
  2014.

\bibitem[Szn91]{Sznitman1991}
Alain-Sol Sznitman.
\newblock Topics in propagation of chaos.
\newblock In {\em \'{E}cole d'\'{E}t\'{e} de {P}robabilit\'{e}s de
  {S}aint-{F}lour {XIX}---1989}, volume 1464 of {\em Lecture Notes in Math.},
  pages 165--251. Springer, Berlin, 1991.

\bibitem[TB04]{TB2004}
Chad~M. Topaz and Andrea~L. Bertozzi.
\newblock Swarming patterns in a two-dimensional kinematic model for biological
  groups.
\newblock {\em SIAM J. Appl. Math.}, 65(1):152--174, 2004.

\bibitem[TBL06]{TBL2006}
Chad~M. Topaz, Andrea~L. Bertozzi, and Mark~A. Lewis.
\newblock A nonlocal continuum model for biological aggregation.
\newblock {\em Bull. Math. Biol.}, 68(7):1601--1623, 2006.

\bibitem[Tos00]{Toscani2000}
Giuseppe Toscani.
\newblock One-dimensional kinetic models of granular flows.
\newblock {\em M2AN Math. Model. Numer. Anal.}, 34(6):1277--1291, 2000.

\bibitem[Vel02]{Velazquez2002}
J.~J.~L. Vel\'{a}zquez.
\newblock Stability of some mechanisms of chemotactic aggregation.
\newblock {\em SIAM J. Appl. Math.}, 62(5):1581--1633, 2002.

\bibitem[Vel04a]{Velazquez2004i}
J.~J.~L. Vel\'{a}zquez.
\newblock Point dynamics in a singular limit of the {K}eller-{S}egel model.
  {I}. {M}otion of the concentration regions.
\newblock {\em SIAM J. Appl. Math.}, 64(4):1198--1223, 2004.

\bibitem[Vel04b]{Velazquez2004ii}
J.~J.~L. Vel\'{a}zquez.
\newblock Point dynamics in a singular limit of the {K}eller-{S}egel model.
  {II}. {F}ormation of the concentration regions.
\newblock {\em SIAM J. Appl. Math.}, 64(4):1224--1248, 2004.

\bibitem[Wei18]{Wei2018}
Dongyi Wei.
\newblock Global well-posedness and blow-up for the 2-{D}
  {P}atlak-{K}eller-{S}egel equation.
\newblock {\em J. Funct. Anal.}, 274(2):388--401, 2018.

\bibitem[Wol33]{Wolibner1933}
W.~Wolibner.
\newblock Un theor\`eme sur l'existence du mouvement plan d'un fluide parfait,
  homog\`ene, incompressible, pendant un temps infiniment long.
\newblock {\em Math. Z.}, 37(1):698--726, 1933.

\bibitem[Yud63]{Yudovich1963}
V.I. Yudovich.
\newblock Non-stationary flow of an ideal incompressible liquid.
\newblock {\em USSR Computational Mathematics and Mathematical Physics},
  3(6):1407 -- 1456, 1963.

\end{thebibliography}
\end{document}